\documentclass{amsart}
\usepackage{etex}
\usepackage{tikz-cd}
\usepackage{amssymb}
\usepackage{latexsym}
\usepackage{amsmath}
\usepackage{enumerate}
\usepackage{amsthm}
\usepackage{wasysym}
\usepackage{stmaryrd}
\usepackage{mathrsfs}
\usepackage{pifont}
\usepackage{tikz}
\usepackage{cases}

\makeatletter
\@ifpackageloaded{fixme}{}{\usepackage[draft]{fixme}} 
\makeatother
\FXRegisterAuthor{av}{aav}{AV}
\FXRegisterAuthor{tmil}{atmil}{TL minor}
\FXRegisterAuthor{tl}{atl}{TL major}
\definecolor{mydarkgreen}{rgb}{0,0.34,0}

\newcommand{\arbop}{\circledcirc}

\newcommand{\lna}[1]{\ensuremath{\mathsf{#1}}}

\newcommand{\takeout}[1]{}
\newcommand{\rfsc}[1]{\S\,\ref{#1}}

\newcommand{\rfse}[1]{\S\,\ref{sec:#1}}

\newcommand{\deq}{:=}

\newcommand{\arsub}{\lna{a}}    

\newcommand{\iA}{\ensuremath{{\sf HLC}^{\sharp}}}   
\newcommand{\iAm}{\ensuremath{{\sf HLC}^{\flat}}}
\newcommand{\Di}{\lna{Di}}

\newcommand{\ia}{\iA}
\newcommand{\iam}{\iAm}
\newcommand{\di}{\Di}

\newcommand{\sba}[1]{\ensuremath{{#1}_{\arsub}}} 
\newcommand{\sbb}[1]{\ensuremath{{#1}_{\opr}}} 

\newcommand{\lsa}[1]{\sba{\lna{#1}}} 
\newcommand{\lsb}[1]{\sbb{\lna{#1}}} 

\newcommand{\ipc}{\lna{IPC}}
\newcommand{\ka}{\lna{K}}

\newcommand{\ik}{\lsb{iK}}

\newcommand{\Boxa}{\lsa{Box}}

\newcommand{\na}{\ensuremath{\sf N_a}}

\newcommand{\quata}{\ensuremath{4_{\sf a}}}

\newcommand{\tr}{{\sf Tr}}

\newcommand{\weak}{{\sf W}}

\newcommand{\pers}{{\sf P}}
\newcommand{\mont}{{\sf M}}

\usepackage{tabularx}
\newcolumntype{L}[1]{>{\raggedright\let\newline\\\arraybackslash\hspace{0pt}}m{#1}}
\newcolumntype{C}[1]{>{\centering\let\newline\\\arraybackslash\hspace{0pt}}m{#1}}
\newcolumntype{R}[1]{>{\raggedleft\let\newline\\\arraybackslash\hspace{0pt}}m{#1}}

\newcommand{\qee} {\hspace*{2mm}\hfill \ding{109}}

\renewcommand{\iff}{\leftrightarrow}
\renewcommand{\phi}{\varphi}
\newcommand{\qedright}{\belowdisplayskip=-12pt}

\definecolor{uuxgreen}{cmyk}{1,0,0.75,0}
\definecolor{uuxred}{cmyk}{0.2,1,0.9,0.1}
\definecolor{uuyblue}  {cmyk}{0.9,0.55,0,0}

\newtheorem{theorem}{Theorem}[section]
\newtheorem{define}[theorem]{Definition}
\newenvironment{definition}{\begin{define} \rm}{\qee\end{define}}
\newtheorem{exa}[theorem]{Example}
\newenvironment{example}{\begin{exa} \rm}{\qee\end{exa}}
\newtheorem{exerc}[theorem]{Exercise}

\newtheorem{conj}[theorem]{Conjecture}

\newtheorem{ques}[theorem]{Open Question}
\newenvironment{question}{\begin{ques} \rm}{\qee\end{ques}}
\newtheorem{lem}[theorem]{Lemma}
\newenvironment{lemma}{\begin{lem} \it}{\end{lem}}
\newtheorem{cor}[theorem]{Corollary}
\newenvironment{corollary}{\begin{cor} \it}{\end{cor}}
\newtheorem{rem}[theorem]{Remark}
\newenvironment{remark}{\begin{rem} \rm}{\qee\end{rem}}

\DeclareMathOperator{\possible}{\text{\tikz[scale=.6ex/1cm,baseline=-.6ex,rotate=45,line width=.1ex]{
                            \draw (-1,-1) rectangle (1,1);}}}
\DeclareMathOperator{\necessary}{\text{\tikz[scale=.6ex/1cm,baseline=-.6ex,line width=.1ex]{
                            \draw (-1,-1) rectangle (1,1);}}}

\DeclareMathOperator{\dotnecessary}{\text{\tikz[scale=.6ex/1cm,baseline=-.6ex,line width=.1ex]{
                            \draw (-1,-1) rectangle (1,1);  \draw[fill=black] (0,0) circle (.25);}}}

\newcommand{\medent}{\medskip\noindent}
 \newcommand{\tupel}[1]{{\langle #1 \rangle}}
\newcommand{\verz}[1]{\{ #1 \}}
\newcommand{\To}{\Rightarrow}

\newcommand{\hyph}{\mbox{-}}

\newcommand{\gn}[1]{{\underline{\ulcorner #1 \urcorner}}}

\newcommand{\opr}{\necessary}
\newcommand{\dotbox}{\dotnecessary}

\newcommand{\oco}{\possible}

\newcommand{\smadent}{\smallskip\noindent}
\newcommand{\kle}{\ll}
\newcommand{\averz}[3]{\{ #1 \in #2 \mid #3 \}}
 \DeclareSymbolFont{symbolsC}{U}{txsyc}{m}{n}
\DeclareMathSymbol{\strictif}{\mathrel}{symbolsC}{74}
\DeclareMathSymbol{\strictfi}{\mathrel}{symbolsC}{75}
\DeclareMathSymbol{\strictiff}{\mathrel}{symbolsC}{76}
\newcommand{\tto}{\strictif}
\newcommand{\ifff}{\strictiff}
\newcommand{\lang}{{\mathbb L}}

\newcommand{\pvsa}{{\widetilde p}}
\newcommand{\pvsb}{{\widetilde q}}
\newcommand{\pvsc}{{\widetilde r}}

\newcommand{\C}{\lsb{S}}

\newcommand{\Lo}{\lna{L}}
\newcommand{\sL}{\lna{sL}}
\newcommand{\UR}{\lna{UR}}
\newcommand{\sU}{\lna{sU}}
\newcommand{\Uz}{\ensuremath{\lna{U}_0}}
\newcommand{\Uu}{\ensuremath{\lna{U}_1}}
\newcommand{\Ud}{\ensuremath{\lna{U}_2}}
\newcommand{\SuR}{\lna{SuR}}
\newcommand{\Suz}{\ensuremath{\lna{Su}_0}}
\newcommand{\Suu}{\ensuremath{\lna{Su}_1}}
\newcommand{\Sud}{\ensuremath{\lna{Su}_2}}
\newcommand{\Sut}{\ensuremath{\lna{Su}_3}}
\newcommand{\quat}{\ensuremath{4_{\opr}}}
\newcommand{\TF}{\ensuremath{\top\lna{F}}}
\newcommand{\TFu}{\ensuremath{\top\lna{F}_1}}
\newcommand{\TFt}{\ensuremath{\top\lna{F}_2}}

\newcommand{\iaff}{\ensuremath{\lna{iA}_\lna{fp}}}

\newcommand{\gl}{\lna{GL}}
\newcommand{\igl}{\lna{iGL}}
\newcommand{\iglu}[1]{\ensuremath{\lna{iGL}(#1)}}

\newcommand{\sysunus}{\lsa{iSL}}
\newcommand{\sysduo}{\ensuremath{{\sf iSL}^+_{\sf a}}}
\newcommand{\systres}{\lsb{iSL}}
\newcommand{\cat}{\ensuremath{\mathbb C}}
\newcommand{\fip}[1]{\digamma #1}

\usepackage{multicol}
\newcommand{\ccon}{\ensuremath{\mathfrak c}}
\newcommand{\comp}[1]{\ccon(#1)}
\newcommand{\cocl}[1]{{\mathbb C}_{#1}}
\newcommand{\bba}[3]{#1\mathrel{\mathcal Z_{#2}}#3}

\newcommand{\refeq}[1]{(\ref{#1})}

\usepackage{mathpartir}

\newcommand{\catC}{\mathcal{C}}
\newcommand{\ibox}{\mbox{\textcolor{uuyblue}{$\opr$}}}
\def\sol#1{#1^\dagger}
\newcommand{\Id}{\mathsf{Id}}

\newcommand{\point}{\mathsf{p}}

\newcommand{\sub}[2]{[#1:#2]}

\newcommand{\id}{\mathsf{id}}
\newcommand{\truem}{\mathtt{true}}
\newcommand{\falsem}{\mathtt{false}}
\newcommand{\ctE}{\catC}
\newcommand{\ctC}{\catC}

\newcommand{\cmp}{\cdot}

\newcommand{\finim}{\,\mathtt{fin}
{\scriptstyle{!}}}
\newcommand{\chrm}[1]{\chi_{#1}}
\newcommand{\trmo}{\mathbf{1}}
\newcommand{\pb}{\text{\large\pigpenfont J}}
\newcommand{\Tconn}[1]{#1}
\newcommand{\topT}{\Tconn{\truem}}
\newcommand{\botT}{\Tconn{\falsem}}
\newcommand{\andT}{\Tconn{\wedge}}
\newcommand{\negT}{\Tconn{\neg}}
\newcommand{\orT}{\Tconn{\vee}}
\newcommand{\impT}{\Tconn{\rightarrow}}

\newcommand{\eqcT}{\mathtt{eq}}
\newcommand{\eqT}{\approx}
\newcommand{\leqcT}{\mathtt{leq}}

\newcommand{\monar}{\rightarrowtail}

\newcommand{\prdf}{{\scriptstyle\Pi}_1}

\newcommand{\prdar}[1]{\langle#1\rangle}
\newcommand{\cprar}[1]{[#1]}
\newcommand{\Prop}{\mathcal{P}}
\newcommand{\PropS}[1]{\mathcal{P}^{\ctC}_{\monar}(#1)}

\newcommand{\logvar}{\varLambda}
\newcommand{\logcon}{\Lambda}
\renewcommand{\preceq}{\preccurlyeq}
\renewcommand{\succeq}{\succcurlyeq}
\newcommand{\nico}[1]{\mathfrak n_{#1}}
\newcommand{\visser}{the first author}
\newcommand{\Visser}{The first author}

\newcommand{\pbmod}[2]{\mbox{\textcolor{mydarkgreen}{$\opr$}}^{\,#1}_{#2}}

\pdfmapline{=dictsym DictSym <dictsym.pfb}
\pdfmapline{=pigpen <pigpen.pfa}

\title{Lewis and Brouwer meet Strong L\"ob}
\author{Albert Visser}
 \address{Philosophy, Faculty of Humanities,
                Utrecht University,
               Janskerkhof 13,
                3512BL~~Utrecht, The Netherlands}
\email{a.visser@uu.nl}
  \author{Tadeusz Litak}
  \address{Informatik 8, FAU Erlangen-N\"{u}rnberg, Germany}
    \email{tadeusz.litak@gmail.com}
\date{\today}



\thanks{We thank Marta Bilkova who provided us with Example \ref{martasmurf}.}

\usepackage{pigpen}  
\usepackage[all]{xy}
\xyoption{2cell}
\xyoption{curve}
\UseTwocells
\SelectTips{cm}{}

\begin{document}

\begin{abstract} 
We study the principle $\phi \to \opr\phi$, known as \emph{Strength} or \emph{the Completeness Principle}, over the constructive version of L\"ob's Logic.  
We consider this principle both for the modal language with the necessity operator $\opr$ and for the modal language
with the Lewis arrow $\tto$, where L\"ob's Logic is suitably adapted. For example, $\opr$ is defined as  $\top \tto (\cdot)$. 

Central insights
of provability logic, like the de Jongh-Sambin Theorem and the de Jongh-Sambin-Bernardi Theorem, take a simple form in the presence of Strength.
We present these simple versions.
We discuss the semantics of two salient systems and prove uniform interpolation for both. In addition, we sketch arithmetical interpretations of our systems. Finally, we 
describe the various connections of our subject with Computer Science.
\end{abstract}

\maketitle


\section{Introduction}

Strength is strange. What counts as strength in one context may appear as weakness in another.
Classical logic is strong. It can prove more than intuitionistic logic. 
However, particularly when the logical vocabulary is extended with new operators, classical logic is not open to a number of possibilities. It walls itself in by excluding too much.
 Principles that trivialise over the classical base can axiomatize important logics over the intuitionistic base. The latter  is more openminded---a form of strength.\footnote{We discuss this further in Section \ref{sec:classical}.}

In this paper we focus on the combination of modality and constructivism. 
There are (at least) two striking examples of modal phenomena that trivialise classically
but are quite significant constructively.

The first is the Lewis Arrow.
 In our
paper \cite{lita:lewi18}, 
we discussed the fact that the interreducibility of the Lewis arrow
and the modal box need not hold in constructive modal logic. This insight opens up
the study of intuitionistic logics with Lewis arrow $\tto$  as a new area of interest.

The second, well-known, phenomenon of intuitionistic modal logic is that the Completeness Principle 
$\C:  \phi\to\opr\phi$
 does not trivialise  over \igl, the constructive version of L\"ob's Logic \gl. 
 Here the `{\sf S}' stands for \emph{strength}.
    
 In the present paper, we study the Lewis
 Arrow in the presence of the Completeness (Strength) Principle and of L\"ob's Principle: the system \sysunus.
Moreover,  we study the $\opr\,$-version of \sysunus\, the system \systres, which can be viewed as an extension.  
 
 \begin{remark}
  Interestingly, the 
 Completeness Principle is one of the few known major axioms whose 
 box variant axiomatises the same constructive logic as its arrow variant 
 $ \lsa{S}:  (\phi\to\psi) \to (\phi\tto\psi)$
  even in the \emph{absence} of the principle $\Boxa$ that collapses the arrow to 
  the box (\cite[Lemma 4.10]{lita:lewi18} and Lemma~\ref{minismurf} of the present paper). So, it truly is also an arrow principle.
\end{remark}

\noindent
The study of Lewisian provability logic in the presence of \C\ allows us to present the central modal insights of
provability logic in their simplest possible forms. These are the de Jongh-Sambin-Bernardi Theorem on the uniqueness
of fixed points and the de Jongh-Sambin theorem on the explicit definability of fixed points. 
Especially, the de Jongh-Sambin Theorem on the existence
of explicit fixed points becomes almost trivial. 
In Section \ref{sec:fixpoints}, we present these results in our context. We also treat some closely related issues, like a  Beth-style Rule,
the reverse mathematics of uniqueness and explicit definability of fixed points, the question what follows if we stipulate
primitive fixed points, and an extension of the class of fixed points beyond the guarded/modalized ones.

Section \ref{ks} instantiates the general technique of ``finitary Henkin'' completeness proofs to our systems, \sysunus\ and \systres,  and discusses the notion of bounded bisimulation.
This provides the spadework for the proof of Uniform Interpolation (Section \ref{unifint}) for \sysunus\ and \systres, which naturally builds upon the (semantic) proof
of Uniform Interpolation for Intuitionistic Propositional Logic. 
Our proof is model theoretic. Recently, Hugo F\'{e}r\'{e}e, Iris van der Giessen, Sam van Gool, and Ian Shillito \cite{HFetal24} 
 announced a Pitts-style syntactic proof of uniform interpolation for \systres. This completes a picture: we can view uniform interpolation here
 both model- and proof-theoretically. 

The question of Uniform Interpolation remains open for any Lewisian systems that lack \C\ and for any modal intuitionistic
provability logics for the $\opr\,$-language that lack \C. 

Both systems  \systres\ and are sound for a variety of arithmetical interpretations (Section \ref{arint}). We briefly sketch the \emph{status questionis} for arithmetical interpretations of our two target systems.
  Zoethout and \visser\ provide an example
of an arithmetical interpretation for which the box system is both sound and complete \cite{viss:prov19}.  No such example is known in the Lewis case.

We end our paper with an extensive discussion of connections between the systems studied in this paper, CS-oriented type theories and categorical logic.


\section{Principles and Logics}\label{basics}
In this section, we introduce various principles and various logics.
Moreover, we provide some of the basic derivations in the logics at hand.  More specifically, Section \ref{langlang} gives the basic definitions of the syntax and languages in question. Section \ref{sec:logic} introduces the basic axioms and rules, and the three main systems of interest make their appearance in Section \ref{clsl}. Subsequently, Section \ref{clofra} treats the class of variable-free formulas. In Section \ref{smurfsmorf}, we develop the machinery of compositional translations and interpretations, which we employ, in Section \ref{smurfzilla}, to compare our three basic systems. We conclude with a short discussion of the ``degenerate'' classical strong L\"ob system in Section \ref{sec:classical}.


\subsection{Languages} \label{langlang}

Suppose we are given modal connectives $\arbop_0$, \dots, $\arbop_{n-1}$ where $\arbop_i$ has arity $\ell_i$. We define
the modal language $\lang(\arbop_0,\dots,\arbop_{n-1})$ by:

\begin{itemize}
\item
$\pi ::= p_0 \mid p_1 \mid \dots$,
\item
$\phi ::= \pi \mid \top \mid \bot \mid  (\phi \wedge \phi) \mid (\phi \vee \phi) \mid (\phi \to \phi) \mid 
\arbop_0(\overbrace{\phi,\dots,\phi}^{\ell_0\times}) \mid \dots $ \\ 
\hspace*{0.8cm} $\dots \mid \arbop_{n-1}(\overbrace{\phi,\dots,\phi}^{\ell_{n-1}\times})$.
\end{itemize}
We will use $p$, $q$, $r$, \dots\ as informal symbols for the propositional variables. 
 We employ the following abbreviations.
\begin{itemize}
\item
$\neg\,\phi := (\phi \to \bot)$,
\item
$(\phi \iff \psi) := ((\phi \to \psi) \wedge (\psi\to \phi))$,
\end{itemize}

Our main languages are 
$ \lang(\tto)$ where $\tto$ is binary and $\lang(\opr)$, where
$\opr$ is unary. We use infix notation for $\tto$.
%
%
In $\lang(\tto)$, we use the following additional abbreviations.
\begin{itemize}
\item
$\phi \ifff \psi := ((\phi\tto \psi) \wedge (\psi\tto \phi))$.
\item
$\opr\phi := (\top\tto \phi)$.
\end{itemize}
In both $\lang(\opr)$ and $\lang(\tto)$, we use the following abbreviation.
\begin{itemize}
\item
$\dotbox \phi := (\phi \wedge \opr\phi)$.
\end{itemize}


\subsection{Logics} \label{sec:logic}
A \emph{logic} $\logvar$ in a language  $\lang(\arbop_0,\dots,\arbop_{n-1})$
  is a set of formulas that contains all substitution instances of the propositional tautologies of
  intuitionistic propositional logic \ipc, and is  closed under modus ponens and substitution.
  In addition, we demand that the modal operators are \emph{extensional}, i.e., for $i<n$, we have:
 if  $\logvar \vdash \bigwedge_{j< \ell_i}(\phi_i\iff \psi_i)$, then 
 $\logvar \vdash \arbop_i(\phi_0,\dots,\phi_{\ell_{i-1}}) \iff   \arbop_i(\psi_0,\dots,\psi_{\ell_{i-1}})$.

\begin{remark}  
As a consequence of this definition, two logics are equal if they are identical as sets of theorems.
\end{remark}

\noindent
Our  basic system \iam\  in the language $\lang(\tto)$ is
 given by the following axioms:
\begin{description}
\item[prop]
axioms and rules for {\ipc}
\item[\tr]
$((\phi \tto \psi) \wedge (\psi \tto \chi)) \to (\phi \tto \chi)$
\item[\ka]
$((\phi \tto \psi) \wedge (\phi \tto \chi)) \to (\phi \tto (\psi\wedge\chi))$
\item[\na]
$\vdash \phi \to \psi \;\;\; \To \;\;\; \vdash \phi \tto \psi$ 
\end{description}

\noindent
One can easily derive the axioms of  {\ik}, the intuitionistic version  of the classical system {\sf K} (without $\oco$),
 for the $\opr\,$-language from \iam.  In this paper, we will mostly start from a stronger base \ia. This is \iam\ plus \di, where
 \di\ is given as follows:
 \begin{description}
 \item[\di]
$((\phi \tto \chi) \wedge (\psi\tto\chi)) \to ((\phi\vee\psi)\tto\chi)$
\end{description}
 
\noindent
We say that $\logvar$ is \emph{normal} if it extends {\iam}.
We say that $\logvar$ is \emph{strongly normal} if it extends {\iA}.
 We will consider the following additional principles. 

\begin{multicols}{2}
\begin{description}
\item[\lsb{4}]
$\opr\phi \to \opr\opr \phi$
\item[\C] 
$\phi \to \opr \phi$
\item[\lsb{\Lo}]
$\opr(\opr\phi \to \phi) \to \opr\phi$
\item[\lsb{\sL}]
$(\opr\phi \to \phi)\to \phi$
\end{description}
\columnbreak
\begin{description}
\item[\lsa{4}]
$\phi \tto \opr \phi$
\item[\pers]
$\phi\tto \psi \to \opr(\phi\tto\psi)$
\item[\lsa{S}] 
$(\phi \to \psi) \to (\phi \tto \psi)$
\item[\lsa{L}]
$(\opr\phi \to \phi) \tto \phi$
\item[\weak]
$((\phi \wedge \opr\psi) \tto \psi) \to (\phi\tto \psi)$
\item[\mont]
$\phi \tto \psi \to (\opr\chi \to \phi)\tto(\opr\chi \to \psi)$
\item[\Boxa]
$((\chi \wedge \phi) \tto \psi) \to (\chi \tto (\phi \to \psi))$
\end{description}
\end{multicols}

We will consider the box principles like $\C$ both in $\lang(\opr)$ and in $\lang(\tto)$.
Of course, to make sense of this, we have to read them in  $\lang(\tto)$ employing the
translation $\opr \;\mapsto\; (\top\tto(\cdot))$.

\subsection{Three Logics}\label{clsl}

In this subsection, we introduce our basic logics \sysunus, \sysduo\  and \systres.
We provide various alternative axiomatisations of \sysunus.
We define the following logics in $\lang(\tto)$.
\begin{itemize}
\item
$\sysunus := \iA + \sL$, 
\item
 $\sysduo := \iA + \sL  + \Boxa$. 
\end{itemize}
We define, in $\lang(\opr)$:
\begin{itemize}
\item
$\systres := \ik + \sL$.
\end{itemize}

We start by showing that \C\ and \lsa{S} are the same principle over \iA. 

\begin{lem}\label{minismurf}
The following systems are equal:
$\iA + \C$ and $\iA + \lsa{S}$.
\end{lem}

This follows from Lemma 4.10 in our paper \cite{lita:lewi18}. We give a full proof for  reader's convenience. 

\begin{proof}
We have:
\begin{eqnarray*}
\iA +\C \vdash \phi \to \psi & \to & \opr(\phi\to \psi) \\
& \to & \phi \tto \psi
\end{eqnarray*}
The second step is by \iA\ reasoning. 

\medent
In the other direction, we have: \qedright
\begin{eqnarray*}
\iA +\lsa{S} \vdash  \psi & \to & \top \to  \psi \\
& \to & \top \tto \psi \\
& \to & \opr\psi
\end{eqnarray*}
\end{proof}

\noindent
On the basis of Lemma~\ref{minismurf},
we will use \C\ and \lsa{S}\ interchangeably. 
We note that, since these arguments are disjunction-free, the result also works when we replace \ia\ by \iam.

\begin{lemma}\label{identitysmurf}
The following systems are the same:\\
\hspace*{2.5cm} $\sysunus =\iA + \sL$, $\iA + \C + \weak$, $\iA + \C + \lsa{\Lo} $, $\iA + \C + \lsb{\Lo}$.
\end{lemma}

This follows from Lemmas 4.12 and 4.13 in our paper \cite{lita:lewi18}. Again, we give a full proof for  reader's convenience. 

\begin{proof}
 We first show $\iA + \sL \vdash \iA + \C + \weak$.
The proof that $\iA + \sL \vdash \C$  is adapted from a proof by Dick de Jongh.
We have
\begin{eqnarray*}
\iA + \sL  \vdash \phi & \to & (\opr(\phi \wedge \opr\phi) \to (\phi \wedge\opr\phi)) \\
& \to & (\phi \wedge \opr\phi) \\
& \to & \opr\phi
\end{eqnarray*}
Next we show that $\iA + \sL \vdash \weak$.
\begin{eqnarray*}
 \iA + \sL \vdash (\phi\wedge\opr\psi) \tto \psi & \to & (\opr(\phi\wedge\opr\psi) \to \opr\psi) \wedge ((\phi\wedge\opr\psi) \tto \psi) \\
 & \to & ((\opr\phi \wedge \opr\opr\psi) \to \opr\psi)   \wedge ((\phi\wedge\opr\psi) \tto \psi)  \\
 & \to & (\opr\phi \to (\opr\opr \psi \to \opr\psi))  \wedge ((\phi\wedge\opr\psi) \tto \psi)  \\
 & \to & (\opr\phi\to \opr\psi)  \wedge ((\phi\wedge\opr\psi) \tto \psi)  \\
 & \to & (\phi \to (\phi \wedge \opr\psi))  \wedge ((\phi\wedge\opr\psi) \tto \psi)  \\
 & \to & (\phi \tto (\phi \wedge \opr\psi))  \wedge ((\phi\wedge\opr\psi) \tto \psi)  \\
 & \to & \phi \tto \psi
\end{eqnarray*}

\medent
We show $ \iA + \C + \weak \vdash \iA + \C + \lsa{\Lo} $. We have
$\iA  \vdash ((\opr\phi \to \phi) \wedge \opr\phi) \tto \phi$, Moreover,
\begin{eqnarray*}
\iA + \C + \weak  \vdash ((\opr\phi \to \phi) \wedge \opr\phi) \tto \phi & \to & (\opr\phi \to \phi) \tto \phi.
\end{eqnarray*}
Combining these facts, we are done.

\medent
We show 
 $\iA + \C + \lsa{\Lo} \vdash  \iA + \C + \lsb{\Lo}$.  
 It is clearly sufficient to show that $\iA + \C + \lsa{\Lo} \vdash   \lsb{\Lo}$.
 This is immediate by the fact that
 \begin{eqnarray*}
\iA   \vdash ((\opr\phi\to\phi) \tto \phi) & \to & \opr (\opr\phi\to \phi) \to \opr\phi
\end{eqnarray*}

\medent
Finally, we show $ \iA + \C + \lsb{\Lo}\vdash \iA + \sL$.
We have:\qedright
\begin{eqnarray*}
 \iA + \C + \lsb{\Lo} \vdash (\opr \phi \to  \phi) & \to & ((\opr\phi \to  \phi) \wedge \opr(\opr\phi \to \phi)) \\
& \to & ((\opr\phi \to  \phi) \wedge \opr\phi) \\
& \to & \phi 
\end{eqnarray*}
\end{proof}

\noindent
We note that all the above arguments are disjunction-free. So they would as well work over \iam\ as base theory.

Clearly, there is more to be said about the relationship between \sysunus, \sysduo and \systres.
We will do so in Subsection~\ref{smurfzilla}.
In order to articulate the basic points, we will need some basic machinery, to wit interpretations, 
that is introduced in Subsection~\ref{smurfsmorf}.
Before, we introduce interpretations,  we have a brief look at the variable-free part of \sysunus. 
As we will see, this takes a particularly simple form.

\subsection{The Closed Fragment of \sysunus}\label{clofra}
We consider the
\textit{closed} or \textit{variable-free} fragment of \sysunus, i.e., the logic restricted to formulas without propositional variables.
We will see that these formulas have particularly simple normal forms, to wit the degrees of falsity that we define below.

Let $\omega^+$ be $\omega\cup\verz{\infty}$.
We give the elements of $\omega^+$ the obvious ordering. 
Moreover, we set $n+\infty = \infty + n = \infty +\infty = \infty$. 
Let $\alpha,\beta$ range over $\omega^+$.

We define the \emph{degrees of falsity} as follows:
\begin{itemize}
\item
 $\opr^{\,0}\bot := \bot$, $\opr^{\,n+1}\bot := \opr\opr^{\,n}\bot$, $\opr^{\,\infty} \bot := \top$.
 \end{itemize}
We have:

\begin{theorem}
The following equivalences are verifiable in \sysunus.
\begin{itemize}
\item
$\bot \iff \opr^{\,0}\bot$, $\top \iff \opr^{\,\infty} \bot$,
\item
$(\opr^{\,\alpha}\bot \wedge \opr^{\,\beta} \bot) \iff \opr^{\,{\sf min}(\alpha,\beta)}\bot$,
\item
$(\opr^{\,\alpha}\bot \vee \opr^{\,\beta} \bot) \iff \opr^{\,{\sf max}(\alpha,\beta)}\bot$,
\item
$(\opr^{\,\alpha}\bot \to \opr^{\,\beta} \bot) \iff \opr^{\,\infty}\bot$, if $\alpha\leq \beta$, and
$(\opr^{\,\alpha}\bot \to \opr^{\,\beta} \bot) \iff \opr^{\,\beta}\bot$, if $\beta < \alpha$,
\item
$(\opr^{\,\alpha}\bot \tto \opr^{\,\beta} \bot) \iff \opr^{\,\infty}\bot$, if $\alpha\leq \beta$, and
$(\opr^{\,\alpha}\bot \tto \opr^{\,\beta} \bot) \iff \opr^{\,\beta+1}\bot$, if $\beta < \alpha$,
\end{itemize}

\noindent
As a consequence, every closed formula of the modal language is equivalent over \sysunus\ to  a 
degree of falsity.
\end{theorem}

\begin{proof}
We  only prove that $\sysunus \vdash
(\opr^{\,\alpha}\bot \tto \opr^{\,\beta} \bot) \iff \opr^{\,\beta+1}\bot$, in case $\beta < \alpha$. \qedright
\begin{eqnarray*}
\sysunus \vdash (\opr^{\,\alpha}\bot \tto \opr^{\,\beta} \bot) & \to & (\opr^{\,\alpha+1}\bot \to \opr^{\,\beta+1} \bot)\\
& \to & (\opr^{\,\beta+2}\bot \to \opr^{\,\beta+1} \bot)\\
& \to & \opr^{\beta+1} \bot \\
& \to & (\opr^{\,\alpha}\bot \tto \opr^{\,\beta} \bot) 
\end{eqnarray*}
\end{proof}

\noindent
It is easy to see that \sysduo\ and \systres\ inherit the normal forms from \sysunus. (In the last case that is \emph{modulo translation}. 
See Subsection~\ref{smurfsmorf}.)

\subsection{Translations and Interpretations}\label{smurfsmorf}
In this subsection, we introduce some further machinery to compare theories.

 A (\emph{compositional}) \emph{translation} $\kappa: \lang(\arbop_0,\dots,\arbop_{n-1}) \to \lang(\arbop'_0,\dots,\arbop'_{n'-1})$ 
is  a function that assigns to each $\arbop_i$ for $i<n$, an  
$\lang(\arbop'_0,\dots,\arbop'_{n'-1})$-formula $\theta_i(p_0,\dots,p_{\ell_i-1})$ with at most
$p_0,\dots,p_{\ell_i-1}$ free.\footnote{\Visser\  \cite{viss:lobs05} considered a more intricate notion of translation that
functions in the presence of a variable-binding operator. We will not develop such ideas in the present paper.}

We define the identity translation and composition of translations in the obvious way. We define  $\phi^\kappa$ by stipulating that $(\cdot)^\kappa$ commutes with the propositional variables  and the non-modal connectives
and that \[(\arbop_i(\phi_0,\dots,\phi_{\ell_i-1}))^\kappa := \theta_i(\phi_0^\kappa,\dots, \phi_{\ell_i-1}^\kappa).\]

Consider logics $\logvar$ in $\lang(\arbop_0,\dots,\arbop_{n-1})$ and $\logvar'$ in
$\lang(\arbop'_0,\dots,\arbop'_{n'-1})$.
An \emph{interpretation} $K:\logvar\to\logvar'$ is given by a translation $\kappa:\lang(\arbop_0,\dots,\arbop_{n-1})\to
\lang(\arbop'_0,\dots,\arbop'_{n'-1})$, such that $\logvar'\vdash \phi^\kappa$, for all $\phi$ such that $\logvar\vdash \phi$.

Consider interpretations $K,M:\logvar \to \logvar'$ based on $\kappa$, resp. $\mu$. 
We count $K$ and $M$ to be the same if for each $i<n$, $\logvar' \vdash \kappa(\arbop_i) \iff \mu(\arbop_i)$.

We define the  identity interpretation and composition of interpretations in the obvious way.
We easily check that, modulo sameness, we have defined a category \cat\ of logics and interpretations. 
We say that two logics are \emph{synonymous} or \emph{definitionally equivalent} if they are isomorphic in \cat.

The framework of interpretations is quite powerful and can be extended to cover results beyond these discussed 
by  \visser\  \cite{viss:lobs05}, 
such as isomorphisms between entire (sub-)lattices of logics or equipollence with respect to any sound semantics. We postpone 
such considerations to a follow-up paper. In the next subsection we illustrate just one application of particular interest here.


\subsection{The Relation between \sysunus, \sysduo\ and \systres}\label{smurfzilla}

We consider the following interpretations.
\begin{itemize}
\item
${\sf Emb}: \sysunus \to \sysduo$ is based in the identical translation ${\sf id}(\tto) = (p_0\tto p_1)$.
\item
${\sf Triv}:\sysduo \to \sysunus$ is based on the translation ${\sf triv}(\tto) := (\top \tto(p_0\to p_1))$.
\item
${\sf LB}:\sysduo \to \systres$ is based on the translation ${\sf lb}(\tto) := \opr(p_0\to p_1)$.
\item
${\sf BL}:\systres \to \sysduo$ is based on the translation ${\sf bl}(\opr) := (\top\tto p_0)$.
\end{itemize}

\noindent It is easy to see that the defining condition of an interpretation is satisfied in these cases.

We say that two theories are \emph{synonymous} if they are isomorphic in our category.


\begin{theorem}\label{ochtendsmurf}
\begin{enumerate}[i.]
\item
\sysduo\ is a retract of \sysunus, as witnessed by {\sf Emb} and {\sf Triv}, i.e, ${\sf Emb} \circ {\sf Triv}= {\sf ID}_{\sysduo}$.
\item
\sysduo\ and \systres\ are synonymous as witnessed by {\sf LB} and {\sf BL}, i.e.,  
${\sf BL}\circ {\sf LB}= {\sf ID}_{\sysduo}$ and ${\sf LB}\circ {\sf BL}= {\sf ID}_{\systres}$
\end{enumerate}
\end{theorem}

\noindent
We leave the trivial proofs to the reader.

In Subsection~\ref{retrosmurf}, we will use the fact that uniform interpolation is preserved over retractions to prove uniform interpolation for
\sysduo\ and \systres.


\subsection{Classical Variant of Strong L\"ob} \label{sec:classical}

 It is easy to see that the principle ${\sf L}_{\opr}$ for $\bot$ is $\neg\neg\, \opr\bot$, which in the classical setting collapses to $ \opr\bot$. In other words, the resulting logic is that of the \emph{verum} operator: anything prefixed  by $\opr$ becomes equivalent to $\top$.

 In the opening paragraph of this paper, 
we claimed that non-triviality of principles like strong L\"ob illustrates that there is strength in intuitionistic 
openmindedness. Naturally, openmindedness as a form of strength may require certain refinements. 
A natural one would be to say that $\Lambda$ is stronger than $\Lambda'$ if there exists a recursive 
``translation function'' $t$ (not necessarily compositional, but satisfying perhaps some additional requirements instead) 
which is sound and faithful, i.e., $\phi \in \Lambda'$ iff $t(\phi) \in \Lambda$. When comparing intuitionistic and classical 
modal logics this way, one can often take one of ``negative translations'': Glivenko, Kuroda, Kolmogorov or suitable variants of 
G\"odel-Gentzen (none of them being compositional in the sense introduced above). See Litak et al. \cite{LitakPR17} for a 
more detailed discussion. Nevertheless, there are examples of propositional systems based on intuitionistic logic whose 
extension with the classical propositional base does not allow any recursive translation function 
\cite[Sec. 4.1]{LitakPR17}.  It is natural to ask if negative translations work as intended at least between \systres\ 
and its classical counterpart. However, the above observation ensures that all such translations work: relatively to the classical extension of strong L\"ob, 
$\opr$ is just the verum operator, so modal subformulas trivialize and the result reduces to the standard relationship between \lna{CPC}\ and \lna{IPC}. A somewhat more general observation is that negative translations generally work for any logic containing the Strength (Completeness) Principle due to the derivability of the modal version of Double Negation Shift  \cite[Sec. 6]{LitakPR17}.


\section{Fixed Points} \label{sec:fixpoints}
This section is devoted to the treatment of the uniqueness of fixed points and of the existence of explicit fixed points in our context. After introducing the necessary concepts (Section \ref{sec:basics}), we discuss a Beth-style rule (Section \ref{beth}), the interaction between strong  uniqueness and strong L\"ob (Section \ref{susl}), explicit definability of fixed points (Section \ref{efp}), the question what follows if we stipulate
primitive fixed points (Section \ref{fpsl}), and an extension of the class of fixed points beyond the guarded/modalized ones (Section \ref{expansivesmurf}).

\subsection{Basics for Substitution} \label{sec:basics}
We first introduce the necessary notational conventions and several substitution principles. In Theorem \ref{substi1}, we verify where these principles hold.

\subsubsection{Notational Conventions} \label{sec:conventions}
Let $\phi$ range over the formulas of $\lang(\tto)$. We introduce the class of formulas $\chi_r$ that are \emph{modalised in $r$}.
Let $\pi_r$ range over the propositional variables except $r$.
We define:
\begin{itemize}
\item
$\chi_r ::= \pi_r \mid \top \mid \bot  \mid (\chi_r \wedge \chi_r) \mid (\chi_r \vee \chi_r) \mid (\chi_r \to \chi_r) \mid \phi \tto \phi$
\end{itemize}

We will also say that $r$ is modalised in $\chi_r$ or that $r$ occurs only on guarded places in $\chi_r$.
Clearly, this means that every occurrence of $r$ is in the scope of an arrow.

Substitution is a central notion in the study of fixed points. We will employ two notational devices for
substitution.

Substitutions $\sigma$ can be viewed as functions that operate on formulas.  We will employ postfix notation
for that, so, $\phi\sigma$ is the result of applying $\sigma$ to $\phi$.
Our substitutions have a finite intended domain. Outside of the
domain the substitution operates as the identity mapping. E.g., the substitution $\sub r\psi$ has intended domain $\verz{r}$ but leaves
the variables distinct from $r$ alone. 

If the variable, say $r$, of substitution is given in context, we will write
 $\phi\psi$ for $\phi\sub r\psi$. We note that $(\phi\psi)\chi$ is equal to $\phi(\psi\chi)$.
 So we may write $\phi\psi\chi$.

\subsubsection{Substitution Principles}
We have a number of substitution principles. We use $r$ as the variable of substitution.\footnote{Strictly speaking,
we need to have the principles for all possible choices of the variable of substitution. Thus, $p$ has to be read as
meta-variable ranging over all choices of concrete variables.}
\begin{description}
\item[\SuR]
$ \vdash \phi \iff \psi \;\; \To \;\;  \vdash \chi\phi \iff \nu\psi$.
\item[\Suz]
$ \opr (\phi \iff \psi) \to \opr (\chi\phi \iff \chi\psi)$ 
\item[\Suu]
$ \dotbox (\phi \iff \psi) \to \dotbox(\chi\phi \iff \chi\psi)$
\item[\Sud]
$ \opr (\phi \iff \psi) \to ( \chi_r\phi \iff \chi_r\psi)$
\item[\Sut]
$(\phi\iff \psi) \to ( \chi\phi \iff \chi\psi)$
\end{description}

\noindent
Here is the obvious theorem on substitution.
\begin{theorem}\label{substi1}
We have:
\begin{itemize}
\item
Any logic $\logvar$ is closed under {\SuR}.
\item
$\ia+{\quat} \vdash \Suz$.
\item
$\ia+{\quat} \vdash \Suu$.
\item
$\ia+{\quat} \vdash \Sud$.
\item
$\ia+\C \vdash \Sut$
\end{itemize} 
\end{theorem} 

\noindent
The great strength of $\ia+\C$ is that substitution behaves as if we are in ordinary non-modal propositional logic.
We  note that, over \ia, each of \quata, \pers\ and \C\ imply \quat, so we have \Suz, \Suu, and \Sud\ in \ia\ plus each of these
principles.
We also note that each of our substitution principles holds when we replace \ia\ by \ik\ and read them in the box-language.

\begin{remark}
By taking, in \Sut, $\phi := r$, $\psi:=\top$,  and $\nu := \opr r$, we see that $\ia+\Sut \vdash \C$.
So, over \ia, the principles \C\ and \Sut\ are interderivable.
\end{remark}

\subsection{A Beth Rule over  $\ia+\C$}\label{beth}
We interpose a brief subsection on (a strong version of) the Beth rule. This does have
a connection with fixed points since Smory\'nski used Beth's Theorem to
prove the existence of explicit fixed points in classical \gl. See  \cite{smor:beth78}. In our context, the detour
over the Beth rule is not needed, since, as we will see, the proof of the existence
of explicit fixed points is very simple.

A logic $\logvar$ is closed under the $\top$-Beth Rule if,
whenever $\vdash_\logvar (\phi p \wedge \phi q) \to (p \iff q)$,
then $\vdash_\logvar \phi p \to  (p \iff \phi \top)$.
Here $p$ is the variable of substitution. We assume that $q$ does not occur in $\phi$.
(Note that $\phi p = \phi$. We exhibit the $p$ just for readability.)

We note that the $\top$-Beth Rule is not an admissible rule in the official sense of the word since it involves
explicit mention of variables.

\begin{theorem}
Suppose $\logvar$ extends $\ia+\C$.
Then, $\logvar$ is closed under the $\top$-Beth Rule. 
\end{theorem}

\begin{proof}
Suppose $\logvar$  extends $\ia+\C$. Let  $p$ is the variable of substitution and suppose  $q$ does not occur in $\phi$.

Suppose $\vdash_\logvar (\phi p \wedge \phi q) \to (p \iff q)$.
Taking $q := \top$, we find \[\vdash_\logvar (\phi p \wedge \phi \top) \to (p \iff \top)\] and, hence, (a)
$\vdash_\logvar \phi p \to (\phi\top \to p)$.
Conversely, by \Sut, we may conclude that (b) $\vdash_\logvar \phi p \to (p \to \phi \top)$.
Combining (a) and (b), we find $\vdash_\logvar \phi p \to (p \iff \phi \top)$.
\end{proof}

\begin{question}
Are there other strongly normal logics than just the extensions of $\ia+\C$ that are closed under the $\top$-Beth Rule?
If so, is there some principle $X$ such that a strongly normal logic $\logvar$ is closed under the Beth Rule 
iff $\logvar$ extends $\ia+X$?
\end{question}

\subsection{Strong Uniqueness meets Strong L\"ob}\label{susl}
In this section, we consider strong uniqueness of guarded fixed points.

In the context of classical {\gl},
the uniqueness of fixed points was proved, independently, by Dick de Jongh (unpublished) 
Giovanni Sambin \cite{samb:effe76} and Claudio
Bernardi in 1974 \cite{bern:uniq76}. Our proof of Theorem~\ref{josabe} is an adaptation 
of the classical proof. 

We formulate some uniqueness principles.\footnote{The classical uniqueness principle of the de Jongh-Sambin-Bernardi result
 is {\Ud}. This is somewhat curious, since the \emph{prima facie} 
 strongest version is \Uz.} We use $r$ as the variable of substitution.
\begin{description}
\item[\UR]
$\vdash (\phi \iff \chi_r\phi) \wedge (\psi \iff \chi_r\psi) \;\;  \To \;\; \vdash (\phi \iff \psi)$
\item[\Uz]
$(\opr(\phi \iff \chi_r\phi) \wedge \opr(\psi \iff \chi_r\psi)) \to \opr(\phi \iff \psi)$ 
\item[\Uu]
$(\dotbox(\phi \iff \chi_r\phi) \wedge \dotbox(\psi \iff \chi_r\psi)) \to \dotbox(\phi \iff \psi)$
\item[\Ud]
$(\dotbox(\phi \iff \chi_r\phi) \wedge \dotbox(\psi \iff \chi_r\psi)) \to (\phi \iff \psi)$
\item[\sU]
$((\phi \iff \chi_r\phi) \wedge (\psi \iff \chi_r\psi)) \to (\phi \iff \psi)$ 
\end{description}

\noindent
Our primary target in this paper is {\sU}. We will look at the other principles in Remarks~\ref{olijkesmurf} and \ref{knutselsmurf}.
We verify  that {\sU} is a logical principle, i.e., that the instances of {\sU} are closed under substitution.

\begin{theorem}
The instances of {\sU} are closed under substitution. 
\end{theorem}

\begin{proof}
Consider an instance $I(r,\chi_r,\phi,\psi)$ of {\sU}. Let $\sigma$ be any substitution for the variables $\mathcal V$ of
$r$, $\chi_r$, $\phi$, $\psi$. 
 Suppose $r'$ occurs neither in the domain nor in the range
of $\sigma$.  
We first show that 
\[\chi_r \sub r\phi\sigma = \chi_r\sub r {r'} \sigma \sub {r'} {\phi\sigma}\text{ and }\chi_r\sub r\psi\sigma = \chi_r \sub r{r'} \sigma \sub{r'}{\psi\sigma}.\]   

\noindent
To see this we prove by induction on formulas $\alpha$ with variables in $\mathcal V$, that, for any formula $\beta$, we have 
$\alpha\sub r\beta\sigma = \alpha\sub r{r'} \sigma \sub{r'}{\beta\sigma}$. The only interesting cases are the variables.
Suppose $p \in \mathcal V$ and $p$ is not $r$, We have:
\[ p\sub r\beta\sigma = p\sigma  =  p \sigma \sub{r'}{\beta\sigma}  = p\sub r{r'} \sigma \sub{r'}{\beta\sigma}.\] 
Moreover,
\[r \sub r\beta\sigma = \beta\sigma =  r' \sub{r'}{ \beta\sigma} = r'\sigma \sub{r'}{ \beta\sigma}  =  r \sub r{r'} \sigma \sub{r'}{ \beta\sigma}.\]

\noindent
It follows that $I(r,\chi_r,\phi,\psi)\sigma$ is the instance $I(r',\chi_r\sub r{ r'}\sigma, \phi\sigma, \psi\sigma)$. Since $r'$ is fresh,
$\chi_r\sub r{r'}\sigma$ is modalised in $r'$.
\end{proof}

\noindent
We show that Strong L\"ob gives Strong Uniqueness.

\begin{theorem}\label{josabe}
$\sysunus \vdash \sU$
\end{theorem}

\begin{proof} 
 Let $\alpha := ((\phi \iff \chi_r\phi) \wedge (\psi \iff \chi_r\psi))$, where $r$ is the variable of substitution.
We have, using {\Sud}: 
\begin{eqnarray*}
\sysunus + \alpha \vdash  \opr(\phi \iff \psi)  & \to & (\chi_r\phi \iff \chi_r\psi) \\
& \to & (\phi \iff \psi)
\end{eqnarray*}
Hence, by {\sL}, we may conclude $\sysunus + \alpha \vdash (\phi \iff \psi)$.
\end{proof}

\begin{remark}\label{olijkesmurf}
By minor adaptations of the above proof, we see that $\ia + \lsb{\Lo} $ proves each of {\Uz}, {\Uu} and {\Ud}.
\end{remark}

\noindent
We prove that, conversely, Strong Uniqueness gives us Strong L\"ob. 

\begin{theorem}\label{sterkesmurf}
$\ia+\sU \vdash \sL$.
\end{theorem}

\begin{proof}
We work in $\ia+\sU$. Let $\phi$ be any formula. Suppose $r$ does not occur in $\phi$. Let
 $\chi_r := (\phi \vee \opr r)$. We have, by {\sU} for $r$, $\chi_r$, $\phi$ and $\top$:
 \[ (((\phi \vee \opr \phi) \iff \phi) \wedge ((\phi \vee \opr \top) \iff \top)) \to (\phi \iff \top).\] 
By {\ia} reasoning, this gives us:  $(\opr \phi\to\phi)\to\phi$. 
\end{proof}
\noindent
We may conclude that {\sysunus} is $\ia+\sU$.

\begin{remark}
We have the analogue of Theorem~\ref{sterkesmurf} for ${\sf HLC}^\flat$. We take $\chi_r := ((\opr r \to \phi) \to \phi)$. By
{\sU} for $r$, $\chi_r$, $\phi$ and $\top$, we find:
 \[ ((((\opr \phi \to \phi) \to \phi) \iff \phi) \wedge (((\opr\top\to\phi)\to\phi) \iff \top)) \to (\phi \iff \top).\] 
 
 \noindent
 By  ${\sf HLC}^\flat$ reasoning, this reduces to:  $(\opr \phi\to\phi)\to\phi$. 
\end{remark}

\begin{remark}\label{knutselsmurf}
What is the relationship between {\UR}, {\Uz}, {\Uu} and {\Ud}? And do they have analogous powers as {\sU} over
{\ia}? More can be said. However, we refrain from doing so here, since the idea of this paper
is to concentrate on the simple case of the Strong L\"ob Principle.
\end{remark}

\subsection{Explicit Fixed Points}\label{efp}

In this subsection we prove the de Jongh-Sambin Theorem in the context of \sysunus.

\begin{remark}
We briefly sketch the well-known history of the de Jongh-Sambin Theorem. 
 As the list below illustrates, the theorem remains an area of active research over
a long period.
\begin{description}
\item[1974]
Dick de Jongh provides unpublished proofs (both a semantical one and a syntactical one)  of the existence of explicit fixed points in {\gl}.  
 \item[1976] A syntactic proof is published by Giovanni Sambin \cite{samb:effe76}, who explicitly states that the intuitionistic propositional base is sufficient. Also, 
 George Boolos found a proof of explicit definability using
characteristic formulas. 
\item[1978]
Craig Smory\'nski proves explicit definability via Beth's Theorem. 
See \cite{smor:beth78}.\footnote{Larisa Maksimova \cite{maks:defi89} shows that, conversely, Beth's 
Theorem follows from the existence
of explicit fixed points. See also \cite{hoog:defi01}.}
\item[1982]
A simplified version of Sambin's approach is provided by Giovanni Sambin and Silvio Valentini \cite{samb:moda82}. This version employs
a proof-theoretical analysis of a sequent calculus for L\"ob's Logic.
\item[1990]
An improved version of Boolos' proof is given by Zachary Gleit and Warren Goldfarb \cite{glei:char90}.
Also, there is a   proof by Lisa Reidhaar-Olson that is close to the proof of Sambin-Valentini \cite{reid:newp90}. 
\item[1991]
Dick de Jongh and Albert Visser \cite{dejo:expl91} give the fixed point calculation for the classical system of interpretability logic {\sf IL}.
They note that the computation simplifies for {\sf ILW}.
\item[1996]
In a subsequent paper \cite{dejo:embe96}, de Jongh and Visser give the very simple fixed point calculation for \systres.
(The system is called {\sf i}-{\sf GL}$\{{\sf C}\}$ in that paper.)
\item[2005]
Rosalie Iemhoff, Dick de Jongh, and Chunlai Zhou give a fixed point calculation for ${\sf iGL}_{\sf a}$ and  ${\sf i GW}_{\sf a}$. 
(They call these systems {\sf iPL}, respectively {\sf iPW}.)
\item[2009]
 Luca Alberucci and Alessandro Facchini provide a proof using the modal $\mu$-calculus \cite{albe:moda09}.
 \end{description}
\end{remark}

\noindent
Let the variable of substitution be $r$. We define:
\begin{description}
\item[\TF]
$\chi_r\top \iff \chi_r\chi_r\top$.
\end{description}
It is easy to see that the instances of {\TF} are closed under substitution. So, {\TF}
can be viewed as a logical principle.

\begin{theorem}\label{grotesmurf}\label{djs}
 $\sysunus \vdash \TF$.
\end{theorem} 

\begin{proof}
We have:
\begin{eqnarray*}
\sysunus \vdash \chi_r\top & \to &  (\chi_r\top \wedge (\top \iff\chi_r\top)) \\
& \to & \chi_r\chi_r \top
\end{eqnarray*}

Conversely,\qedright
\begin{eqnarray*}
\sysunus \vdash \chi_r\chi_r\top& \to & (\opr \chi_r \top \to \chi_r\top) \\
& \to & \chi_r\top 
\end{eqnarray*}
\end{proof}

\noindent
What happens for weaker theories than \sysunus?
We refer the reader to \cite{iemh:prop05} and \cite{lita:lewi19} for more information.

We prove a converse of Theorem~\ref{grotesmurf}.
\begin{theorem}\label{bricolatorsmurf}
$\ia+\TF \vdash \sL$.
\end{theorem}

\noindent
So \sL\ and \TF\ are interderivable over \ia.

\begin{proof}
Suppose $r$ does not occur in $\phi$. Let $\chi_r := (\opr r \to \phi)$.
We have, applying {\TF} to $\chi_r$:
\[ \ia+\TF \vdash (\opr \top \to \phi) \iff (\opr (\opr \top \to \phi)  \to \phi) .\] 
By {\ia} reasoning this gives $\ia+\TF \vdash (\opr\phi \to \phi ) \to \phi$.
\end{proof}

\begin{remark}
We note that, in this argument, the competing fixed point equations for the formulas
$\chi'_r := \opr(r\to \phi)$ and $\chi''_r := r \tto\phi$, where $r$ does not occur in $\phi$, do not deliver the goods.
 This can be seen by noting that the corresponding fixed points can be obtained as a 
 instances of  \TFu, respectively \TFt, where:
\begin{description}
\item[\TFu]
$\psi \top \iff \psi\psi\top$, where $\psi$ is of the form $\opr\nu$.
\item[\TFt]
$\psi \top \iff \psi\psi\top$, where $\psi$ is of the form $\rho\tto \nu$.
\end{description}
One can show that these principles do not yield \sL. One quick way to see this is to note that
 $\ia+\TFu+\Boxa = \ia+\TFt +\Boxa$ and that $\ia+\TFu+\Boxa$ is definitionally equivalent to
 $\igl$, the constructive version of \gl. For more information, see \cite{lita:lewi19}. 
\end{remark} 

We summarise the results of the Sections~\ref{clsl}, \ref{susl} and \ref{tfsl} about \sysunus:

\begin{theorem}
The following are axiomatisations of \sysunus:
\begin{itemize}
\item
$\iA + \sL$, 
\item
$\iA + \C + \weak$, 
\item
$\iA + \C + \lsa{\Lo} $, 
\item
$\iA + \C + \lsb{\Lo}$,
\item
$\ia+\sU$,
\item
$\ia+\TF$.
\end{itemize}
Here, in all cases, \C\ can  be replaced by \lsa{S}.
\end{theorem}

\subsection{Primitive Fixed Points}\label{fpsl}\label{tfsl}
In Subsection~\ref{efp}, we proved a converse, Theorem~\ref{bricolatorsmurf}, of our
version of the de Jongh-Sambin Theorem, i.e. Theorem~\ref{grotesmurf}. Here we used
the specific explicit form of the fixed points. What happens if, instead, the fixed points are primitively added,
so that we cannot make use of information on what the fixed point looks like?
In that case, we also have a converse, but we do need \C\ in order to derive \sL.\footnote{That \C\ is really needed can be seen from the
fact that (i) \sL\ gives \C\ over \ia, and (ii) the possibility to extend finite tree-models of \ia\ with fixed points.}

We study the addition of primitive fixed points for modalised formulas to $\ia+\C$. We show that the added fixed points
give us the Strong L\"ob Principle. The reasoning is a minor adaptation of L\"ob's classical proof in \cite{loeb:solu55}.

We need to define the notion of a fixed point logic for explicit guarded fixed points. There are three (main) options (as far as we can see). 
We can take each fixed point just a constant.
If we do that however, we do lose the substitution property, so the resulting system is not really a logic. 
Secondly, we can add a fixed point operator that binds variables. 
The advantage of the variable-binding approach is that we can make visible that taking fixed points commutes with substitution.
This seems to be an important insight. In the case at hand, where we study Strong L\"ob, this insight is immediate. It also holds
more generally, for weaker forms of L\"ob's Principle, but in those cases to check it is more delicate. 
Thirdly we can develop the theory on cyclic syntax. (This third option was explored for the case of classical provability logic without
the principle 4 in \cite{viss:cycl21}.)

In the present paper, we will follow a middle road related to the second option: we add a fixed point operator that can only bind one designated
variable that we write as $\ast$. To study the addition of a variable binding operator in the full sense,
one needs to develop some machinery that is beyond the scope of the present paper. We hope to return to these matters in
a subsequent paper.

We prove \sL\ using the addition of primitive fixed points as constants. 
We define system $\iaff$. The language of \iaff\ is given by $\phi$ in the following grammar:
{\small
\begin{itemize}
\item
$\pi ::= p_0 \mid p_1\mid p_2 \dots$, 
\item
$\phi ::=  \pi \mid \top \mid \bot \mid \fip\chi \mid (\phi \wedge \phi) \mid   (\phi \vee \phi) \mid (\phi\to\phi) \mid (\phi\tto \phi)$,
\item
$\psi ::=  \ast \mid \pi \mid \top \mid \bot \mid \fip\chi \mid (\psi \wedge \psi) \mid   (\psi \vee \psi) \mid (\psi\to\psi) \mid (\psi\tto \psi) $,
\item
$\chi ::= \pi \mid \top \mid \bot \mid \fip\chi  \mid (\chi \wedge \chi) \mid (\chi \vee \chi) \mid (\chi \to \chi) \mid (\psi \tto \psi)$.
\end{itemize}
}

\noindent
We note that $\ast$ is supposed to be bound by the lowest $\fip{}$ that is above it in the parse tree.
 We are allowed to substitute $\phi$-formulas in $\phi$-, $\psi$ and $\chi$-formulas in the usual way.  

The reasoning system of \iaff\ is given as follows.
\begin{itemize}
\item We have the axioms of \ia\ extended to the new language. 
\item We add as axioms
 fixed point equations: $\fip\chi \iff \chi \fip\chi$, where $\ast$ functions as the variable of substitution.
\item We add necessitation for the extended system.
\end{itemize}
 
 \begin{theorem}\label{slsmurf}
$\iaff+{\C} \vdash \sL$
\end{theorem}
\begin{proof}
We reason in $\iaff+\C$. Suppose (\dag) $\opr\phi \to \phi$. Let $\ell := \fip{\opr(\ast \to \phi)}$.
Suppose $\ell$. It follows, by {\C}, that $\opr\ell$ and, by the fixed point equation, that $\opr(\ell \to \phi)$. Hence, $\opr\phi$. By
(\dag), we have $\phi$. We cancel our assumption $\ell$, to obtain (\ddag) $\ell \to \phi$. It follows that $\opr(\ell \to \phi)$. Hence,
again by the fixed point equation, $\ell$. Combining this with (\ddag), we find $\phi$.  
\end{proof}

\noindent
It is easy to see that $\ell':= \fip{(\opr \ast \to \phi)}$ gives us an alternative proof.

We note that via Strong L\"ob we have explicit and unique fixed points, so:
 \[ \iaff+\C \vdash \fip \chi \iff \chi\top.\]
 
 \begin{remark}
 We can extend the notion of interpretation, along the lines of \cite{viss:lobs05}, in such 
 a way that one can show that \sysunus\ and $\iaff+\C$ are isomorphic
 in the resulting category.
 \end{remark}

\subsection{Extending The Class of Fixed Points} \label{expansivesmurf}
In the classical case, we can extend the fixed point property to a wider class of formulas. 
See \cite{bent:moda06} and \cite{viss:lobs05}. We extend the idea to \sysunus.

Let us say that $p$ occurs only \emph{semi-positively} in $\phi$ if it all its 
occurences are positive or in the scope of a
$\tto$. We have:

\begin{theorem}
Let $p$ be the variable of substitution.
Suppose $p$ occurs only semi-positively in $\phi$. Then,
$\sysunus \vdash \phi \top \iff  \phi\phi\top$. 
Moreover, $\phi\top$ is the maximal fixed point for $\top$.
\end{theorem}

\begin{proof}
We note that by \Sut, we have $\sysunus \vdash \psi \top \to  \psi\psi\top$, for any $\psi$.
So, we have this direction, \emph{a fortiori}, for $\phi$.

Now suppose $p$ occurs only positively in $\psi$. We have $\sysunus \vdash \psi \top \to \psi\psi\top$ and,
hence, by positivity, $\sysunus \vdash \psi\psi \top \to \psi\psi\psi\top$. (This step requires an easy induction that we
leave to the reader.)
By \Sut, we find 
$\sysunus \vdash \psi\psi \top \to \psi\top$, and, thus, we have the fixed point property for $\psi$, to wit  $\sysunus \vdash \psi\top \iff \psi\psi \top$.

We partition the occurrences of $p$ in $\phi$ in two groups in such a way that all occurrences
in the first  group are positive and all occurrences in the second group are in the scope of $\tto$.
Let $p_0$ and $p_1$ be fresh variables and let $\widetilde \phi(p_0,p_1)$ be the result of replacing
the occurences of $p$ in the first group by $p_0$ and the occurrences of $p$ in the second group by
$p_1$. We have:
\begin{eqnarray*}
\sysunus +  \phi\phi\top  & \vdash & \opr \phi \top \to \widetilde \phi (\phi\top, \top) \\
& \vdash  &   \opr \phi \top \to \widetilde \phi (\top, \top) \\
& \vdash &  \opr \phi \top \to  \phi ( \top) \\
& \vdash & \phi\top
\end{eqnarray*}

\noindent The second step of the above argument uses the fixed point property for $\widetilde \phi(p_1,\top)$ w.r.t. the variable $p_1$.

To see that $\phi\top$ is maximal, let $\alpha$ be another fixed point.
We have:\qedright
\begin{eqnarray*}
\sysunus + \alpha &  \vdash & (\alpha \iff \top) \wedge \phi\alpha \\
& \vdash & \phi\top 
\end{eqnarray*}
\end{proof}

\begin{remark}
We can also prove the existence of a minimal fixed point using a result of Wim Ruitenburg  \cite{ruit:peri84,ghil:heyt20} or using uniform interpolation.
We can use these results to interpret a version of the constructive  $\mu\nu$-calculus in \sysunus.
\end{remark}


\section{Kripke Completeness and Bounded Bisimulations}\label{ks}


In the present section we introduce Kripke semantics for our systems \sysunus\ and \sysduo. After presenting the basic definitions (Section \ref{badsmurf}), we instantiate the general ``finitary Henkin construction'' \cite{iemh:prov01,iemh:prop05} 
to prove completeness of our systems (Section \ref{fhc}). In Section \ref{bobbysmurf}, we introduce the notion of bounded bisimulation, which will be a main tool to prove uniform interpolation.

\subsection{Basic Definitions} \label{badsmurf}
We treat the Kripke semantics for \sysunus\ and \sysduo. 
We  use $\preceq$ for the  weak partial ordering associated with intuitionistic logic and 
$\prec$ for the corresponding strict ordering. Moreover, we use $\sqsubset$ for the ordering associated to
the Lewis arrow. We employ the following semantics for \sysunus. 

A \emph{frame} $\mathcal F$ is a triple
$\tupel{K,\preceq,\sqsubset}$. Here:
\begin{itemize}
\item
 $K$ is a finite, non-empty set.
 \item
 $\preceq$ is a weak partial order on
$K$.
\item
$\sqsubset$ is an irreflexive relation on $K$. We demand further:
\begin{description}
\item[strong]
$\sqsubset$ is a subrelation of $\preceq$,
\item[${\tto}\hyph\mathbf p$]
If $x \preceq y \sqsubset z$, then $x \sqsubset w$.
\end{description}
\end{itemize}

\noindent
We note that it follows that $\sqsubset$ is transitive. Thus, $\sqsubset$ is a strict partial order.

Let $\mathcal P$ be a set of propositional variables.
A \emph{$\mathcal P$-model} $\mathcal K$ is a quadruple  $\tupel{K,\preceq,\sqsubset,\Vdash_0}$.
 The relation $\Vdash_0$ relates the nodes in $K$ with propositional
variables in $\mathcal P$. We demand that, if $x \preceq y$ and $x\Vdash_0 p$, then $y\Vdash_0 p$.
We define the forcing relation $\Vdash$ for $\mathcal K$ as follows:
\begin{itemize}
\item
$k\Vdash p$ iff $k\Vdash_0 p$;
\item
$k\nVdash \bot$;
\item
$k\Vdash \top$;
\item
$k \Vdash \psi \wedge \chi$ iff $k\Vdash \psi$ and $k\Vdash \chi$;
\item
$k \Vdash \psi \vee \chi$ iff $k\Vdash \psi$ or $k\Vdash \chi$;
\item
$k \Vdash \psi \to \chi$ iff, for all $m\succeq k$, if $m\Vdash \psi$, then $m\Vdash \chi$;
\item
$k \Vdash \psi \tto \chi$ iff, for all $m\sqsupset k$, if $m\Vdash \psi$, then $m\Vdash \chi$.
\end{itemize}
It is easy to verify that we have persistence for $\Vdash$:  if $x \preceq y$ and $x\Vdash \phi$, then $y\Vdash \phi$.
We note that ${\tto}\hyph\mathbf p$ is needed in the induction step for the persistence of $\tto$.

We will often omit the designated set of variables $\mathcal P$. Usually, this is simply
the set of all variables. However, in the proof
of uniform interpolation, it will be convenient to keep track of the variables of the model at hand.
 
We easily verify the soundness of our models for \sysunus.

The principle \Boxa\ corresponds to the following frame condition:
\begin{description}
\item[brilliancy]
Suppose $x \sqsubset y \preceq z$, then $x\sqsubset z$.
\end{description}

\noindent
Thus, brilliant models are sound for \sysduo.

\subsection{The Finitary Henkin Construction} \label{fhc}
In this section, we present  (finitary) Henkin constructions for \sysunus\ and \sysduo.  For those who are familiar with
the Henkin constructions for L\"ob's Logic and for intuitionistic logic, what happens here is rather predictable.
For a more extended presentation, see e.g. \cite{iemh:prov01}. 

\subsubsection{The Henkin Construction for \sysunus}
Let us call a set of formulas $X$ \emph{adequate} if it is closed under subformulas and contains $\bot$ and $\top$.
Clearly, every finite set of formulas has a finite
adequate closure.

Let $X$ be a finite adequate set. 
We construct the Henkin model $\mathbb H_X$ on $X$ as follows. 
Our nodes are the $\varDelta\subseteq X$ that are \emph{$X$-prime}, i.e.,
\begin{enumerate}[a.] 
 \item
 $ \varDelta \nvdash_{\sysunus}\bot$,
 \item
 Suppose $\phi \in X$ and $\varDelta\vdash_{\sysunus} \phi$, then $\phi\in \varDelta$,
 \item
 Suppose $(\psi\vee\chi) \in X$ and $\varDelta\vdash_{\sysunus} (\psi\vee \chi)$, then $\psi\in \varDelta$ or $\chi\in \varDelta$.
\end{enumerate}

\noindent
We note that (a), (b) and (c) are really one principle: if $\varDelta$ proves a finite disjunction in \sysunus, then it contains
one of the disjuncts. Here the empty disjunction is $\bot$. 

We take the Henkin model to be a $\mathcal P$-model, where
$\mathcal P$ is the set of propositional variables in $X$.

We take:
\begin{itemize}
\item
$\varDelta\preceq \varDelta'$ iff $\varDelta \subseteq \varDelta'$;
\item
$\varDelta \sqsubset \varDelta'$ iff $\varDelta \neq \varDelta'$ and, for all $\psi\in X$, if 
$\varDelta \vdash_{\sysunus} \bigwedge \varDelta'\tto\psi$, then $\psi\in \varDelta'$;
\item
$\varDelta\Vdash_0 p$ iff $p\in\varDelta$.
\end{itemize}

\begin{theorem}
$\mathbb H_X$ is a $\mathcal P$-model.
\end{theorem}

\begin{proof}
Clearly, $\preceq$ is a weak partial ordering. Suppose $\varDelta\sqsubset \varDelta'$ and $\phi\in \varDelta$.
Then, $\varDelta \vdash_{\sysunus} \bigwedge \varDelta'\tto \phi$. Ergo, $\phi \in \varDelta'$.
So, $\Delta \subseteq \Delta'$ and, thus, we have strength.

Clearly, $\sqsubset$ is irreflexive. Suppose $\varDelta \preceq \varDelta' \sqsubset\varDelta''$.
It follows, by irreflexivity and strength, that $\varDelta \subseteq \varDelta' \subset \varDelta''$.
Ergo, $\varDelta \neq \varDelta''$.  Suppose that $\phi\in X$ and
 $\varDelta \vdash_{\sysunus} \bigwedge\varDelta'' \tto \phi$. Since $\varDelta\subseteq \varDelta'$,
 we find $\varDelta' \vdash_{\sysunus} \bigwedge\varDelta'' \tto \phi$, and, hence,
 $\phi\in\varDelta''$. We may conclude that $\varDelta\sqsubset \varDelta''$.
 \end{proof}
 
 \begin{theorem}
Let $\phi\in X$. We have
$\varDelta \Vdash \phi$ iff $\phi\in \varDelta$.
\end{theorem}

\begin{proof}
 The proof is by induction of $\phi$.
There are just two non-trival cases: the case of implication and the case of the Lewis arrow.
The case of implication is the same as in the proof of the analogous result for intuitionistic propositional logic. We treat the case of
$\tto$.

Suppose $(\psi\tto\chi) \in \varDelta$ and $\varDelta\sqsubset\varDelta'$ and $\varDelta'\Vdash \psi$.
By the induction hypothesis, we have $\psi\in \varDelta'$. So, it follows that $\varDelta \vdash_{\sysunus}  \bigwedge\varDelta'\tto \chi$.
Hence, by the definition of $\sqsubset$, we find that $\chi\in \varDelta'$. Again, by the induction hypothesis,
we have $\varDelta'\Vdash \chi$.

Suppose $(\psi\tto\chi) \in X$ and $(\psi\tto\chi) \not\in \varDelta$. Since,
we have ${\sf W}_{\tt a}$ in \sysunus, it follows that  $\varDelta \nvdash_{\sysunus} (\psi\wedge \opr\chi)\tto\chi$, and,
hence, \emph{a fortiori}, $\varDelta \nvdash_{\sysunus}(\psi\wedge (\psi\tto\chi))\tto\chi$.
We fix some enumeration $\phi_0,\dots \phi_{N-1}$ of $X$.
Let $\varDelta'_0 := \verz{\psi,(\psi \tto \chi)}$. Suppose we have constructed $\varDelta'_n$, for $n< N$.
We define:
\[ 
\varDelta'_{n+1} := \left \{ \begin{array}{ll}
\varDelta'_n \cup\verz{\phi_n} & \text{if $\varDelta \nvdash_{\sysunus} (\phi_n \wedge\bigwedge \varDelta'_n) \tto \chi$} \\
\varDelta_n' & \text{otherwise}
\end{array} \right.
\]

\noindent
We take $\varDelta' := \varDelta'_N$. By induction we find that $\varDelta\nvdash_{\sysunus} \bigwedge\varDelta'\tto \chi$.

Suppose $\phi\in X$. We claim that (\dag) $\phi \in \varDelta'$ iff $\varDelta\nvdash_{\sysunus} (\phi\wedge\bigwedge\varDelta')\tto \chi$.
the left-to-right direction is trivial. Suppose $\varDelta\nvdash_{\sysunus} (\phi\wedge\bigwedge\varDelta')\tto \chi$.
Let $\phi = \phi_i$.  We find that  $\varDelta\nvdash_{\sysunus} (\phi\wedge\bigwedge\varDelta'_i)\tto \chi$.
So $\phi=\phi_i\in\varDelta'_{i+1}\subseteq \varDelta'$.

We verify that $\varDelta'$ is $X$-prime. We just treat case (c), the other cases being simpler.
Suppose $ (\sigma \vee \tau)\in X$ and $\varDelta'\vdash (\sigma\vee\tau)$.
 We note that, by \ia-reasoning, we have:
 \[ \varDelta \vdash_{\sysunus} (((\sigma\vee \tau) \wedge \bigwedge\varDelta')\tto\chi) \iff
 (((\sigma \wedge \bigwedge\varDelta')\tto\chi) \wedge ((\tau \wedge \bigwedge\varDelta')\tto\chi)).\]
 It follows that $\varDelta\nvdash_{\sysunus} \bigwedge((\sigma\vee\tau)\wedge\varDelta')\tto\chi$ iff 
   $\varDelta\nvdash_{\sysunus} (\sigma\wedge \bigwedge\varDelta')\tto\chi$ or $\varDelta\nvdash_{\sysunus} (\tau \wedge \bigwedge\varDelta')\tto\chi$.
  Since $\varDelta'\vdash (\sigma\vee\tau)$, we find that $\varDelta\nvdash_{\sysunus} \bigwedge((\sigma\vee\tau)\wedge\varDelta')\tto\chi$
  and, hence, we have, by (\dag), $\sigma\in \varDelta'$ or $\tau\in \varDelta'$.

We show that $\varDelta \sqsubset \varDelta'$. First $\varDelta \neq \varDelta'$, since $(\psi\tto\chi) \not\in \varDelta$ and  $(\psi\tto\chi)\in \varDelta'$.
Secondly, suppose $\varDelta \vdash_{\sysunus} \bigwedge \varDelta' \tto \nu$. By (\dag), we find $\nu\in\varDelta'$.
 
 We have $\psi\in\varDelta'$ and $\chi\not\in\varDelta'$. Hence, by the induction hypothesis, $\varDelta'\Vdash \psi$ and
 $\varDelta'\nVdash \chi$, and we are done.
 \end{proof}

\noindent
We give an important theorem that is the obvious analogue of a similar fact about Intuitionistic Propositional Logic.

\begin{theorem}\label{priyes}
Let $X$ be a finite adequate set.
Suppose $\varDelta$ is $X$-prime. Then, $\varDelta$ has the disjunction property, i.e., whenever
$\varDelta \vdash_{\sysunus} \phi \vee \psi$, then $\varDelta\vdash_{\sysunus}\phi$ or $\varDelta\vdash_{\sysunus} \psi$.
\end{theorem}

\begin{proof}
Let $X$ be a finite adequate set and let $\varDelta$ be $X$-prime. 
Suppose $\varDelta \vdash _{\sysunus}\phi\vee\psi$, but $\varDelta\nvdash _{\sysunus}\phi$ and  $\varDelta\nvdash_{\sysunus}\psi$.
Consider a Kripke model $\mathcal K$ with root $k$ such that $k\Vdash \varDelta$ and $k\nVdash \phi$.
Consider a Kripke model $\mathcal M$ with root $m$ such that $m\Vdash \varDelta$ and $m\nVdash \phi$.
Let $\mathcal H$ be the upwards closed submodel of $\mathbb H_X$ with root $\varDelta$.
We add a new root $b$ under (the disjoint sum of) $\mathcal H$, $\mathcal K$ and $\mathcal M$.
This root will be $\preceq$ below every node of the original models. It will be $\sqsubset$-below every
node that is $\sqsubset$-above the root of one of the original models.\footnote{The three models involved
might be models on different sets of propositional variables. We adapt the models to models in all propositional variables 
by the minimal action of reading the $\Vdash_0$ as relations between nodes and all variables.}

We stipulate that $b\Vdash p$ iff $p\in \varDelta$. We show that, for $\chi\in X$, 
we have $b\Vdash \chi$ iff $\chi\in \varDelta$. Suppose
$b\Vdash \chi$. Then, by persistence $\varDelta\Vdash \chi$, and, thus, by
the properties of the Henkin model, $\chi\in \varDelta$. We show by induction
on $\chi\in X$, that, if $\chi\in\varDelta$, then $b\Vdash \chi$.

We treat the case of $\to$.
Suppose $b\nVdash \sigma\to\tau$, where $(\sigma\to\tau)\in X$. In case, $b\Vdash \sigma$,  it
follows that $b\nVdash \tau$, and, hence, by the induction hypothesis that $\sigma\in\varDelta$ and $\tau\not\in\varDelta$.
So, \emph{a fortiori}, $(\sigma\to\tau)\not\in\varDelta$.
If $b\nVdash \sigma$, it follows that, for some $c\succ b$, $c\Vdash \sigma$ and $c\nVdash \tau$.
Clearly, $c\succeq b'$, for some root $b'$ of our original three models. It follows that
$b'\nVdash\sigma\to\tau$, and, since $b'\Vdash\varDelta$, $(\sigma\to\tau)\not\in\varDelta$.

 The case of $\tto$ is similar to the second half of the case of $\to$. Surprisingly, the case
 of $\tto$ does not use the induction hypothesis.

 We may conclude that $b\Vdash \varDelta$, but $b\nVdash \phi$ and
$b\nVdash \psi$. A contradiction.
\end{proof}

\subsubsection{The Henkin Construction for \sysduo}

We extend the notion of subformula to sub${}^+$formula, where we also count
$(\psi \to \chi)$ as a sub${}^+$formula of $(\psi\tto\chi)$.
Let us call a set of formulas $X$ \emph{adequate${}^+$} if it is closed under sub${}^+$formulas and contains $\bot$ and $\top$.
Clearly, every finite set of formulas has a finite
adequate${}^+$ closure.

Let $X$ be a finite adequate${}^+$ set. 
We construct the Henkin model $\mathbb H^+_X$ on $X$ as follows. 
Our nodes are the $\varDelta\subseteq X$ that are \emph{$X$-prime}.
We take:
\begin{itemize}
\item
$\varDelta\preceq \varDelta'$ iff $\varDelta \subseteq \varDelta'$;
\item
$\varDelta \sqsubset \varDelta'$ iff
$\varDelta\subset \varDelta'$ and,
for all $(\nu \tto \rho)\in \varDelta$, we have $(\nu \to\rho)\in \varDelta'$;
\item
$\varDelta\Vdash_0 p$ iff $p\in\varDelta$.
\end{itemize}

\begin{theorem}
$\mathbb H^+_X$ is a $\mathcal P$-model.
\end{theorem}

\noindent
We leave the easy proof to the industrious reader.

\begin{theorem}
Suppose $\phi\in X$.
We have $\phi\in \varDelta$ iff $\varDelta\Vdash \phi$.
\end{theorem}

\begin{proof}
We treat  the case of the induction where $\phi = (\psi\tto\chi)\in X$.
Suppose $(\psi\tto\chi) \in \varDelta$ and $\varDelta \sqsubset\varDelta'\Vdash \psi$.
By the induction hypothesis,  $\psi\in\varDelta'$. Since  $\varDelta \sqsubset\varDelta'$,
we have $(\psi\to \chi)\in \varDelta'$. Hence, $\chi\in \varDelta'$. It follows by the induction hypothesis that
$\varDelta' \Vdash \chi$.

Suppose $(\psi\tto\chi) \not\in \varDelta$. It follows that $\varDelta \nvdash_{\sysunus} (\phi \wedge\opr\chi)\tto \chi$, and,
\emph{a fortiori}, $\varDelta \nvdash (\phi \wedge (\psi\tto\chi))\tto \chi$.
Let $\varDelta'_0:= \varDelta\cup\verz{\psi, (\psi\tto\chi)} \cup \verz{(\nu \to \rho)\mid (\nu \tto\rho)\in \varDelta}$.
We can now show that $\varDelta \nvdash_{\sysduo} \bigwedge \varDelta'_0 \tto\chi$. Hence,
$\varDelta_0'\nvdash_{\sysduo} \chi$.
We extend $\varDelta_0'$ in the usual way to a $X$-prime set $\varDelta'$ such that $\chi\not\in \varDelta'$.
One easily sees that $\varDelta\sqsubset \varDelta'$, $\varDelta'\Vdash \psi$, and $\varDelta'\nVdash \chi$.
\end{proof}

\subsubsection{Pre-Henkin Models}
In the proof of uniform interpolation, we will need the notion of pre-Henkin model.
This is simply defined like Henkin model minus the clauses that insure irreflexivity.

Let $X$ be a finite adequate set. For \sysunus, we define the \emph{pre-Henkin model}
 $\mathbb H^\circ_X$ as follows.
 \begin{itemize}
 \item
 The nodes of our model are the $X$-prime subsets $\varDelta$;
\item
$\varDelta\preceq \varDelta'$ iff $\varDelta \subseteq \varDelta'$;
\item
$\varDelta \sqsubset \varDelta'$ iff, for all $\psi\in X$, if 
$\varDelta \vdash_{\sysunus} \bigwedge \varDelta'\tto\psi$, then $\psi\in \varDelta'$;
\item
$\varDelta\Vdash_0 p$ iff $p\in\varDelta$.
\end{itemize}
Inspection shows that we still have strength and ${\tto}\hyph\mathbf p$. Moreover, we still have:
 if $\phi\in X$, then  $\varDelta \Vdash\phi $ iff $\phi\in\varDelta$.

Let $X$ be a finite adequate${}^+$ set. For \sysduo, we define  the pre-Henkin${}^+$ model $\mathbb H^{\circ +}$  by:
\begin{itemize}
 \item
 The nodes of our model are the $X$-prime subsets $\varDelta$;
\item
$\varDelta\preceq \varDelta'$ iff $\varDelta \subseteq \varDelta'$;
\item
$\varDelta \sqsubset \varDelta'$ iff $\varDelta \subseteq \varDelta'$ and, for all $(\psi\tto\chi)\in \varDelta$, we
have $(\psi\to\chi)\in \varDelta'$;
\item
$\varDelta\Vdash_0 p$ iff $p\in\varDelta$.
\end{itemize}
Inspection shows that we still have strength and ${\tto}\hyph\mathbf p$ and brilliance. 
Moreover, if $\phi\in X$, then  $\varDelta \Vdash\phi $ iff $\phi\in\varDelta$.

We will employ the notion ${\sf d}_X(\varDelta)$ of depth in a pre-Henkin model.
This is simply the length of the longest $\prec$-path upwards from $\varDelta$ in 
$\mathbb H^{\circ}_X$. In other words, ${\sf d}_X(\varDelta) := {\sf sup}\verz{{\sf d}_X(\varDelta') \mid \varDelta\prec \varDelta'}$.

\subsection{Bounded Bisimulations} \label{bobbysmurf}
We will use $\pvsa,\pvsb,\pvsc$ to range over finite sets of variables. 
We write ${\sf Sub}^{(+)}(\phi)$ for the set of sub${}^{(+)}$formulas of $\phi$ and we write
${\sf PV}(\phi)$ for the set of proposional variables in $\phi$.

We define the complexity
$\comp{\phi}$ of $\phi$ by recursion on $\phi$ as follows:
\begin{itemize}
\item
$\comp{p} := \comp{\bot} :=\comp{\top} := 0$,
\item
$\comp{\psi\wedge \chi} := \comp{\psi \vee \chi} := {\sf max}(\comp{\psi},\comp{\chi})$,
\item
$\comp{\psi\to \chi} := \comp{\psi\tto \chi} :=  {\sf max}(\comp{\psi},\comp{\chi})+1$,
\end{itemize}

We define:
\begin{itemize}
\item
$\lang(\pvsa\,) := \verz{\phi\in \lang \mid {\sf PV}(\phi) \subseteq \pvsa\,} $;
\item
  $\cocl{n}(\pvsa\,) := \verz{\phi\in \lang(\pvsa) \mid \comp{\phi} \leq n}$.
  \end{itemize}
  The set $\cocl{n}(\pvsa\,)$ is closed under  sub${}^+$formulas. Moreover, it is easy to see that, modulo
  provable equivalence, it is finite. Hence, we can find a finite  adequate${}^+$ set that represents
  $\cocl{n}(\pvsa\,)$ modulo provable equivalence. We will, on occasion, confuse $\cocl{n}(\pvsa\,)$ with such a set of
  representatives.
  
 Let $\omega^+$ be $\omega\cup\verz{\infty}$. We extend the ordering on $\omega$ be stipulating $n<\infty$.
 We take $n+\infty := \infty+n := \infty+\infty := \infty$.
 Consider models $\mathcal K$ with domain $K$ and $\mathcal M$ with domain $M$. \emph{A bounded bisimulation $\mathcal Z$ for $\pvsa$
  between $\mathcal K$ and $\mathcal N$}
 is a ternary relation between $K$, $\omega^+$ and $M$. We write $\bba k\alpha m$ for $\mathcal Z(k,\alpha,m)$.
 We often consider $\mathcal Z$ as an indexed set of binary relations
 $\mathcal Z_\alpha$. 
 Let $\kle$ range over $\preceq$ and $\sqsubset$. 
 We demand:
 \begin{itemize}
 \item
 if $\bba k\alpha m$, then $k\Vdash p$ iff $m\Vdash p$, for all $p\in \pvsa$;
 \item 
  if $\bba k{\alpha+1} m$ and $k\kle k'$, then there is an $m'\gg m$ with $\bba {k'}\alpha {m'}$\\ (zig property);
  \item
   if $\bba k{\alpha+1} m$ and $m\kle m'$, then there is an $k'\gg k$ with $\bba {k'}\alpha{m'}$\\ (zag property).
 \end{itemize}
 
 \noindent
 Note that,  by our conventions, e.g., the zig property for the case $\alpha=\infty$ is:
   if $\bba k{\infty} m$ and $k\kle k'$, then there is an $m'\gg m$ with $\bba {k'}\infty {m'}$.
   Note also that, e.g., the empty set is also a bounded bisimulation.
 
 A bounded bisimulation $\mathcal Z$ is \emph{downward closed} iff, whenever $\alpha \leq \beta$ and $\bba k\beta m$, we have
 $\bba k\alpha m$. Clearly, we can always extend a bounded bisimulation to a downward closed one.
 
 Consider a bounded bisimulation where all triples have the form $\tupel{k,\infty,m}$.
 If we omitt the second components we obtain a binary relation which is \emph{a bisimulation}.
A bisimulation is officially defined as follows. 
 \begin{itemize}
 \item
 if $ k\mathrel{\mathcal B} m$, then $k\Vdash p$ iff $m\Vdash p$, for all $p\in \pvsa$;
 \item 
  if $k \mathrel{\mathcal B} m$ and $k\kle k'$, then there is an $m'\gg m$ with $k'\mathrel{\mathcal B} m'$ (zig property);
  \item
   if $k\mathrel{\mathcal B}m$ and $m\kle m'$, then there is an $k'\gg k$ with $k'\mathrel{\mathcal B} m'$ (zag property).
 \end{itemize}
 
 \noindent
 We note that by replacing all pairs $\tupel{k,m}$ by $\tupel{k,\infty,m}$ we obtain a bounded bisimulation.
 So, bisimulations can be viewed as a subclass of the bounded bisimulations.
 
 It is easy to see that bounded bisimulations and bisimulations for a fixed $\pvsa$ are closed under unions.
 We will write:
 \begin{itemize}
 \item
  $k\simeq_{\pvsa}m$ for: there exists a bisimulation $\mathcal B$ between $\mathcal K$ and
 $\mathcal M$ for $\pvsa$ such that $k\mathrel{\mathcal B} m$;
 \item
 $k\simeq_{\alpha,\pvsa}m$ for: there exists a bounded  bisimulation $\mathcal Z$ between $\mathcal K$ and
 $\mathcal M$ for $\pvsa$ such that $\bba k\alpha m$. 
 \end{itemize}
Clearly, $\simeq_{\pvsa}$ is the maximal bisimulation on $\pvsa$ and $\simeq_{(\cdot),\pvsa}$ is the maximal bounded bisimulation on $\pvsa$.

 \begin{theorem}
 Suppose $\mathcal Z$ is a bounded bisimulation for $\pvsa$ between $\mathcal K$ and $\mathcal M$. 
   We have:
   \begin{enumerate}[a.]
   \item
   if $\bba k\infty m$ and $\phi\in \lang(\pvsa\,)$, then $k\Vdash\phi$ iff $m\Vdash \phi$.
   \item
    if $\bba knm$ and $\phi\in \cocl{n}(\pvsa\,)$, then $k\Vdash\phi$ iff $m\Vdash \phi$.
   \end{enumerate}
   \end{theorem}

 \begin{proof}
 We just prove (b). 
 We proceed by induction on $\phi$. The cases of atoms, conjunction and disjunction are trivial.
 The cases of $\to$ and $\tto$ are analogous. We treat the case of $\tto$.
 Suppose $\phi$ is $\psi\tto \chi$. Since, $\phi\not\in\cocl{0}(\pvsa\,)$, we only need to consider the case
 where the complexity is a successor.
 Suppose $\bba k{n+1}m$ and  $\phi\in\cocl{n+1}(\pvsa)$. In this case, $\psi$ and
 $\chi$ will be in $\cocl{n}(\pvsa\,)$.
 
 Suppose $k\Vdash \psi\tto \chi$ and $m\sqsubset  m' \Vdash \psi$. Then, for some $k'\sqsupset k$, 
 we have $\bba {k'}n{m'}$. It follows, by the induction hypothesis,
  that $k'\Vdash \psi$ and, hence,
 $k'\Vdash \chi$. Since  $\bba {k'}n{m'}$, again by the induction hypothesis, we may
  conclude that $m'\Vdash \chi$.  Hence,  $m\Vdash \psi\tto \chi$.
 The other direction is similar.
 \end{proof}
   
   \begin{theorem}
   Consider two models $\mathcal K$ and $\mathcal M$ and consider
   a node $k$ of $\mathcal K$ and a node $m$ of $\mathcal M$.
   Suppose, for all $\phi\in \cocl{n}(\pvsa\,)$, we have
   $k\Vdash \phi$ iff $m\Vdash \phi$. Then $k\simeq_{n,\pvsa} m$.
   \end{theorem}
   
   \begin{proof}
   We define the relation $\bba k\alpha m$ iff  $\alpha\in\omega$ and, for all $\phi\in \cocl{\alpha}(\pvsa\,)$, we have
   $k\Vdash \phi$ iff $m\Vdash \phi$. It is sufficient to check that $\mathcal Z$ is a bounded bisimulation.
   
   The atomic case is easy. We prove the zig-property for $\mathcal Z_{n+1}$.
   
   \medent
   Suppose $\bba k{n+1}m$ and $k\preceq k'$. We define:
   \begin{itemize}
   \item
   $\theta^+ := \bigwedge \verz{\phi\in \cocl{n}(\pvsa\,)\mid k'\Vdash \phi}$;
   \item
     $\theta^- := \bigvee \verz{\psi\in \cocl{n}(\pvsa\,)\mid k'\nVdash \psi}$.
     \end{itemize}
     Then, $k \nVdash \theta^+\to\theta^-$. Since $(\theta^+\to\theta^-)\in\cocl{n+1}(\pvsa)$,
     we find $m \nVdash \theta^+\to\theta^-$. It follows that there is an
     $m'\succeq m$ with $m'\Vdash \theta^+$ and $m'\nVdash \theta^-$.
     This implies that $\bba{k'}n{m'}$.
     
     Suppose $\bba k{n+1}m$ and $k\sqsubset k'$. We define $\theta^+$ and $\theta^-$ as before.
       It follows that $k \nVdash \theta^+\tto\theta^-$. Since $(\theta^+\tto\theta^-)\in\cocl{n+1}(\pvsa\,)$,
     we find $m \nVdash \theta^+\tto\theta^-$. Thus, there is an
     $m'\sqsupset m$ with $m'\Vdash \theta^+$ and $m'\nVdash \theta^-$.
     This implies that $\bba{k'}n{m'}$.
   \end{proof}

We will need the notion of bisimulation $\pvsb$-extension.
Let $\pvsa$ and $\pvsb$ be disjoint sets of variables.
We consider  a $\pvsa$-model $\mathcal K$  with root $k_0$ and a $\pvsa,\pvsb$-model $\mathcal M$  with root $m_0$.
We say that $\mathcal M$ is a \emph{bisimulation $\pvsb$-extension} of
$\mathcal K$ if $k_0$ is $\pvsa$-bisimilar with $m_0$.


\section{Uniform Interpolation} \label{unifint}

In this section, we prove uniform interpolation for our central systems. We begin with formulating the main result, with the remainder of the section devoted to its proof. 

Let $\logvar$ be one of \sysunus\ and \sysduo. We write $|X|$ for the cardinality of $X$. We have:

\begin{theorem}[Uniform Interpolation] \label{unii}
Consider any formula $\phi$ and any finite set of variables $\pvsb$. Let 
\[\nu:= |\averz{\psi}{{\sf Sub}(\phi) }
{\psi \text{ is a variable, an implication or a  Lewis implication}}|  \]
We have the following.
\begin{enumerate}[1.]
\item
There is a formula $\exists \pvsb\; \phi$, the \emph{post-interpolant} of $\phi$, such that:
     \begin{enumerate}[a.]
     \item
     ${\sf PV}(\exists\pvsb \; \phi) \subseteq  {\sf PV}(\phi)\setminus \pvsb$
    \item
     For all $\psi$ with ${\sf PV}(\psi)\cap  \pvsb=\emptyset$, we have:  
    \[\logvar\vdash  \phi \rightarrow  \psi 
    \;\; \Leftrightarrow \;\;  \logvar\vdash  \exists  \pvsb\;\phi  \rightarrow  \psi.\]
    \item
     $\comp{\exists\pvsb \; \phi} \leq 2 \nu +2$.
    \end{enumerate}
\item 
 There is a formula $\forall \pvsb\; \phi$, the \emph{pre-interpolant} of $\phi$, such that:
         \begin{enumerate}[a.]
         \item
         ${\sf PV}(\forall \pvsb \; \phi) \subseteq  {\sf PV}(\phi)\setminus \pvsb$
         \item
        For all $\psi$ with ${\sf PV}(\psi)\cap  \pvsb=\emptyset$, we have:
        \[ \logvar\vdash  \psi \rightarrow  \phi 
         \;\;\Leftrightarrow\;\;  \logvar\vdash \psi \rightarrow  \forall \pvsb \; \phi.\]
         \item
          $\comp{\forall \pvsb \; \phi} \leq 2\nu  +1$.
        \end{enumerate}
\end{enumerate}
\end{theorem}

\noindent
We note that uniform interpolation strengthens interpolation, since the uniform interpolants just depend on one of the
two formulas involved in interpolation. In classical logic the post- and the pre-interpolant are interdefinable.
In constructive logic, the post-interpolant is definable from the pre-interpolant by taking 
 \[
 \exists \pvsb \; \phi :\iff \forall  r\, (\forall \pvsb \,( \phi \to r) \to r),
 \]
where $r$ is a fresh variable. However, we cannot, \emph{prima facie},  generally define a pre-interpolant from a post-interpolant. 

\begin{example} \label{martasmurf}
The following counter-example was suggested to us by Marta Bilkova. The positive fragment of \ipc\ has post-interpolation,
but not pre-interpolation. The fact that the positive fragment has post-interpolation was shown by Dick de Jongh and Zhiguang Zhao in their article
\cite{dejo:posi15}. The fragment cannot have post-interpolation, since, for example, $\forall p \; p$ defines $\bot$.
\end{example}

\begin{remark}[Historical Note]
The idea of uniform interpolation for \ipc\ was anticipated by Wim Ruitenburg \cite[Example 2.6]{ruit:peri84} and proved by Andrew Pitts \cite{pitt:inte92} 
 using
proof systems allowing efficient cut-elimination  developed
independently by  J\"org Hudelmaier \cite{hude:boun89} and Roy Dyckhoff \cite{dyck:cont92}.
Silvio Ghilardi and Marek Zawadowski \cite{ghil:shea95}, and, independently but somewhat later,
 \visser\  \cite{viss:laye96}, found a model-theoretic proof for Pitts' result using bounded bisimulations.
Both Ghilardi {\&}\ Zawadowski and  \visser\  establish the connection between uniform interpolation and bisimulation quantifiers.

The corresponding result for {\gl} was proved  by Volodya Shavrukov \cite{shav:suba93} using the method of characters as
developed by Zachary Gleit and Warren Goldfarb, who employed them to provide alternative proofs of the Fixed Point Theorem 
 and the ordinary Interpolation Theorem \cite{glei:char90}. The methods of Gleit \& Goldfarb and later of Shavrukov can
be viewed as model-theoretic. The papers  \cite{ghil:shea95} and  \cite{viss:laye96} also provide a new proof for \gl.

While completing the final draft,  F\'{e}r\'{e}e et al. \cite{HFetal24} have announced a Pitts-style syntactic proof of uniform interpolation for \systres. 
Furthermore, their development is supported by a formalization in the Coq proof assistant, which furthermore allows extraction of  verified code computing interpolants. 
\end{remark}

\noindent
We first prove our result for \sysunus. In the Subsection~\ref{microsmurf}, we sketch how to adapt the proof to
the case of \sysduo. In Subsection~\ref{retrosmurf}, we show how do derive the result for \sysduo\ and \systres\
from the result for \sysunus.

The proof of uniform interpolation in Subsection~\ref{uniunus} is a direct adaptation of the proof in \cite{viss:laye96} for the case of
{\sf IPC}.

\subsection{Uniform Interpolation for \sysunus}\label{uniunus}
We prove a central lemma. We define, for a given Kripke model $\mathcal K$ and a $\mathcal K$-node $k$:
 ${\sf Th}_X(k) := \verz{\phi\in X\mid k \Vdash \phi}$.

\begin{lemma}\label{unijoin} Consider pairwise disjoint finite sets of propositional variables $\pvsb$, 
$\pvsa$ and $\pvsc$. Let $X\subseteq {\lang}(\pvsb,\pvsa\,)$
 be a finite adequate set with $\pvsa\subseteq X$. Let $\mathcal K$ be a $\pvsb,\pvsa$-model with root $k_0$, and let
$\mathcal M$ be a  $\pvsa, \pvsc$-model with root $m_0$. we define:{\small
\[ \nico X  := |\averz{\gamma}{X}{\text{$\gamma$ is a propositional variable, an implication or a Lewis implication}}|. \]
}
Suppose that $k_0 \simeq _{2\nico X  +1,\pvsa}m_0$. Then 
there is a bisimulation $\pvsb$-extension $\mathcal N$  of $\mathcal M$ with root $n_0$ such that 
${\sf Th}_X(n_0) = {\sf Th}_X(k_0)$.      
\end{lemma}

\begin{proof}
Let $K$ be the set of nodes of $\mathcal K$ and let  $M$ be the set of nodes of $\mathcal M$.
 Let ${\mathcal Z}$ be a downwards closed witness of 
$k_0 \simeq_{2\nico X  +1,\pvsa} m_0$.
We define, for $k$ in $\mathcal K$: 
${\sf d}_X(k) := {\sf d}_{\mathbb H^\circ_X}({\sf Th}_X(k))$. Note that ${\sf d}_X(k)\leq \nico X $, since
every time we go strictly up in $\mathbb H^\circ_X$, the higher node must contain an extra propositional variable or an
extra implication or an extra Lewis implication.

Consider a pair $\tupel{ \varDelta ,m}$, where $\varDelta$ is in $\mathbb H^\circ_X$ and $m$ is in $\mathcal M$. 
We say that  $k',k,m'$ is a {\em witnessing triple} for 
$\tupel{ \varDelta ,m}$ if: 
\[\varDelta ={\sf Th}_X (k)={\sf Th}_X (k'),\; k'\preceq k,\; m'\preceq m,\; 
k'{\mathcal Z}_{2  {\sf d}_X(k'){+}1}m',\;k{\mathcal Z}_{2 {\sf d}_X(k')}m.\]

\medskip
\[
\begin{tikzcd}[column sep = large]
  \varDelta \arrow[mapsfrom]{r}{{\sf Th}_X}  & 
 k  \arrow[-]{r}{{\mathcal Z}_{2{\sf d}_X(k')}} & 
 m \\
 \varDelta \arrow[mapsfrom]{r}[swap]{{\sf Th}_X} \arrow{u}{=}  & k'  \arrow[-]{r}[swap]{{\mathcal Z}_{2{\sf d}_X(k')+1}}\arrow{u}{\preceq}
  &   m' \arrow{u}[swap]{\preceq}
\end{tikzcd} 
\]

%

%

\medskip
We define the $\pvsb,\pvsa,\pvsc$-model $\mathcal N$
\begin{itemize}
\item 
$N := \verz{\tupel{ \varDelta ,m} \;|\; 
\mbox{there is a witnessing triple for }\tupel{ \varDelta ,m}}$,
\item 	
$n_0 := \tupel{ {\sf Th}_X(k_0),m_0}$,
\item 
$\tupel{ \varDelta ,m} \preceq_{\mathcal N} \tupel{ \varDelta' ,m'} 
\; :\Leftrightarrow \;  \varDelta \preceq_{\mathbb H^\circ_X}\varDelta'
\text{ and } m\preceq _{\mathcal M}m'$,
\item 
$\tupel{ \varDelta ,m} \sqsubset_{\mathcal N} \tupel{ \varDelta' ,m'} 
\; :\Leftrightarrow \;  \varDelta \sqsubset_{\mathbb H^\circ_X}\varDelta'
\text{ and } m\sqsubset _{\mathcal M}m'$,
\item 
$\tupel{ \varDelta ,m} \Vdash _{\mathcal N}a \; :\Leftrightarrow \;  
\varDelta \Vdash_{\mathbb H^\circ_X}a
\mbox{ or }m\Vdash_ {\mathcal M}a$, where $a\in \pvsa,\pvsb,\pvsc$.
\end{itemize}
We easily check that the new accessibility relations have the desired properties for an \sysunus-model.
We note that the irreflexivity
of $\sqsubset_{\mathcal M}$ is sufficient to yield the irreflexivity of $\sqsubset_{\mathcal N}$.
By assumption $k_0{\mathcal Z}_{2\nico X{+}1}m_0$. Moreover, we have
$2  {\sf d}_X(k_0){+}1 \leq 2  \nico X {+}1$.
Hence: $k_0{\mathcal Z}_{2{\sf d}_X(k_0){+}1}m_0$. So, we can take 
$k_0,k_0,m_0$ as witnessing triple  for $n_0$. Thus, indeed, $n_0\in N$.

Let $k',k,m'$ be a witnessing triple  for $\tupel{ \varDelta ,m}$. We note that,
for $p\in \pvsa$,
\[\varDelta \Vdash p \;\;\Leftrightarrow \;\;  k\Vdash p \;\; \Leftrightarrow \;\;  m\Vdash p,\]
and hence: \[\tupel{ \varDelta ,m} \Vdash p  \;\;\Leftrightarrow  \;\; \varDelta \Vdash p  \;\;\Leftrightarrow \;\;  m\Vdash p.\]

\noindent
We claim:
\begin{description}
\item[Claim 1]  $n_0 \simeq_{\pvsa,\pvsc} m_0$.
\item[Claim 2]  For $\phi \in X: \tupel{ \varDelta ,m} \Vdash \phi  
\; \Leftrightarrow \;  \phi \in\varDelta$.
\end{description}
Evidently, the lemma is immediate from the claims.

\medent
{\em We prove Claim 1.} Take, as $\pvsa,\pvsc$-bisimulation, the relation $\mathcal B$ with 
$\tupel{ \varDelta ,m} \mathrel{\mathcal B} m$. 
Clearly, we have
${\sf Th}_{\pvsa,{\pvsc}}(\tupel{ \varDelta ,m}) = {\sf Th}_{\pvsa, \pvsc}(m)$.
Moreover, ${\mathcal B}$ has the zig-property.
We check that ${\mathcal B}$ has the zag-property. 

Let $\kle$ be either $\preceq$ or $\sqsubset$.
Suppose 
$\tupel{ \varDelta ,m}\mathrel{\mathcal B}m \kle s$. We are looking for a pair 
$\tupel{ \varGamma ,s}$ in $N$ such that $\varDelta \kle \varGamma$. 

Let $k',k,m'$
 be a witnessing triple  for $\tupel{ \varDelta ,m}$.  We note that also $m' \kle s$, either by the transitivity $\preceq$ or by $\tto$-p. Since, 
$k' \mathrel{{\mathcal Z}_{2{\sf d}_X(k')+1}} m' \kle s$, there is an $\ell$ such that 
$k'\kle \ell \mathrel{{\mathcal Z}_{2{\sf d}_X(k')}} s$. We take $\varGamma :={\sf Th}_X(\ell)$. 
We will produce a witnessing triple $\ell',\ell,s'$  for $\tupel{\varGamma,s}$.
We distinguish two possibilities.

First,  $\varDelta =\varGamma$. In this case we can take: $\ell' := k'$,
and  $s' :=m'$.  This yields a witnessing triple, since $\kle$ is a subrelation of $\preceq$.
We note that, in case $\kle$ is $\sqsubset$, we have $\varDelta \sqsubset\varDelta$ in $\mathbb H^\circ$, because $k' \sqsubset \ell$.

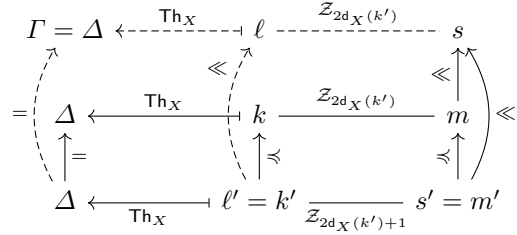
\[
\begin{tikzcd}[column sep = large]
\varGamma = \varDelta  \arrow[mapsfrom, dashed]{r}{{\sf Th}_X} & \ell  \arrow[-, dashed]{r}{{\mathcal Z}_{2{\sf d}_X(k')}}& s \\
  \varDelta \arrow[mapsfrom]{r}{{\sf Th}_X}  & 
 k  \arrow[-]{r}{{\mathcal Z}_{2{\sf d}_X(k')}} & 
 m \arrow{u}{\kle} \\
 \varDelta \arrow[mapsfrom]{r}[swap]{{\sf Th}_X} \arrow{u}[swap]{=}\arrow[dashed, bend left]{uu}{=} & 
 \ell'= k'  \arrow[-]{r}[swap]{{\mathcal Z}_{2{\sf d}_X(k')+1}}\arrow{u}[swap]{\preceq} \arrow[dashed, bend left]{uu}[near end]{\kle}
  &   s'= m' \arrow{u}{\preceq}\arrow[bend right]{uu}[swap]{\kle}
\end{tikzcd} 
\]

%


\medskip\noindent
Secondly,  $\varDelta \neq\varGamma$. In this case we can take:
 $\ell' := \ell $, $s':=s$. To see this, note that, since $k'\kle \ell $, we have: 
$\varDelta = {\sf Th}_X (k')\prec \varGamma$. 
Ergo ${\sf d}_X(\ell )< {\sf d}_X(k')$. It follows that:
$2{\sf d}_X(\ell)+1 \leq 2{\sf d}_X(k')$.
So, $\ell \mathrel{{\mathcal Z}_{2  {\sf d}_X(\ell')+1}} s$ (and, by downward closure, also $\ell \mathrel{{\mathcal Z}_{2{\sf d}_X(\ell')}} s$).
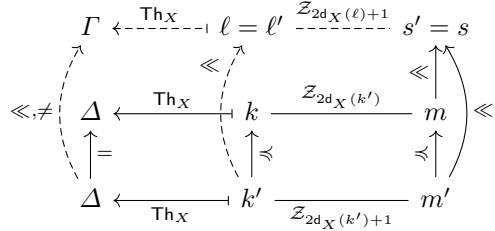
\[
\begin{tikzcd}[column sep = large]
\varGamma   \arrow[mapsfrom, dashed]{r}{{\sf Th}_X} & \ell= \ell'  \arrow[-, dashed]{r}{{\mathcal Z}_{2{\sf d}_X(\ell)+1}}& s'=s \\
  \varDelta \arrow[mapsfrom]{r}{{\sf Th}_X}  & 
 k  \arrow[-]{r}{{\mathcal Z}_{2{\sf d}_X(k')}} & 
 m \arrow{u}{\kle} \\
 \varDelta \arrow[mapsfrom]{r}[swap]{{\sf Th}_X} \arrow{u}[swap]{=}\arrow[dashed, bend left]{uu}{\kle,\neq} & 
 k'  \arrow[-]{r}[swap]{{\mathcal Z}_{2{\sf d}_X(k'){+}1}}\arrow{u}[swap]{\preceq} \arrow[dashed, bend left]{uu}[near end]{\kle}
  &   m' \arrow{u}{\preceq}\arrow[bend right]{uu}[swap]{\kle}
\end{tikzcd} 
\]


\medskip\noindent
Finally, clearly, $n_0 \mathrel{\mathcal B} m_0$.

\smadent {\em We prove Claim 2.} The proof is by induction on $X$. 
The cases of atoms, conjunction and disjunction are trivial. We treat implications and Lewis implications. 
Consider the node 
$\tupel{ \varDelta ,m}$ with witnessing triple  $k',k,m'$.

Let $\phi = (\psi \to \chi)$.
Suppose $\varDelta  \nVdash \psi\rightarrow \chi$. 

If we have $\varDelta \Vdash \psi$ and $\varDelta  \nVdash \chi$, it follows
by the induction hypothesis that $\tupel{\varDelta, m} \Vdash \psi$ and   $\tupel{\varDelta, m}  \nVdash \chi$. Hence,
$\tupel{\varDelta, m}  \nVdash \psi \to \chi$.

Suppose $\varDelta \nVdash \psi$. 
We have  $k  \nVdash \psi \to \chi $. Hence, for some $\ell$ with $k \preceq \ell$, we have $\ell \Vdash \psi$ and $\ell \nVdash \chi$.
Let $\varGamma :={\sf Th}_X (\ell)$. We find:
$\varDelta\prec\varGamma$ and, thus, ${\sf d}_X(k) >  {\sf d}_X(\ell) $.   We note that it follows that $2{\sf d}_X(k')\geq 2$. Since,
 $k \mathrel{{\mathcal Z}_{2 {\sf d}_X(k')}} m$ and $k\preceq \ell$, there is an $s\succeq m$ with
$\ell \mathrel{{\mathcal Z}_{2  {\sf d}_X(k')-1}} s$. Moreover, $2  {\sf d}_X(\ell )+1 \leq 2  {\sf d}_X(k')-1$.
Ergo, $\ell\mathrel{{\mathcal Z}_{2  {\sf d}_X(\ell)+1}}s$. So, $\ell,\ell,s$ is a witnessing triple for 
$\tupel{ \varGamma,s}$. Clearly, $\tupel{ \varDelta ,m}\preceq \tupel{ \varGamma,s}$. 
By the Induction Hypothesis: $\tupel{ \varGamma,s}\Vdash \psi$ and 
$\tupel{ \varGamma,s} \nVdash \chi$. Hence,
$ \tupel{ \varDelta ,m} \nVdash \psi \rightarrow \chi $.

\[
\begin{tikzcd}[column sep = large]
\varGamma   \arrow[mapsfrom,dashed]{r}{{\sf Th}_X} & \ell= \ell'  \arrow[-, dashed]{r}{{\mathcal Z}_{2{\sf d}_X(\ell)+1}}& s'=s \\
  \varDelta \arrow[mapsfrom]{r}{{\sf Th}_X}\arrow[dashed]{u}{\prec}  & 
 k  \arrow[-]{r}{{\mathcal Z}_{2{\sf d}_X(k')}} \arrow[dashed]{u}{\preceq}& 
 m \arrow[dashed]{u}[swap]{\preceq} \\
 \varDelta \arrow[mapsfrom]{r}[swap]{{\sf Th}_X} \arrow{u}{=}& 
 k'  \arrow[-]{r}[swap]{{\mathcal Z}_{2{\sf d}_X(k'){+}1}}\arrow{u}{\preceq}
  &   m' \arrow{u}[swap]{\preceq}
\end{tikzcd} 
\]


\medskip\noindent Conversely, suppose $\tupel{ \varDelta ,m} \nVdash \psi\rightarrow \chi$. There is a 
$\tupel{ \varGamma ,s}$ in $\mathcal N$ with 
$\tupel{ \varDelta ,m}\preceq \tupel{ \varGamma ,s}$ and 
$\tupel{ \varGamma ,s}\Vdash \psi$ and $\tupel{ \varGamma ,s} \nVdash \chi $. 
Clearly $\varDelta \preceq \varGamma$. By the Induction Hypothesis 
$\varGamma \Vdash \psi$ and $\varGamma  \nVdash \chi$. Ergo,
$\varDelta  \nVdash \psi\rightarrow \chi$. 

\medent
We turn to the case that $\phi = ( \psi \tto \chi)$. Suppose $\varDelta  \nVdash \psi \tto \chi$. It follows that $k  \nVdash
 \psi \tto \chi$. Hence, there is an $\ell \sqsupset k$ such that $\ell \Vdash \psi$ and $\ell \nVdash \chi$. We choose $\ell$ to be $\sqsubset$-maximal with
this property. So, we have $\ell \Vdash \psi \tto \chi$. Let $\varGamma := {\sf Th}_X(\ell)$. We have $\varGamma \sqsupset \varDelta$ and
$\varGamma \neq \varDelta$. So ${\sf d}_X(k) > {\sf d}_X(\ell)$. Since,
 $k \mathrel{{\mathcal Z}_{2  {\sf d}_X(k')}} m$ and $k\sqsubset \ell$, there is an $s\sqsupset m$ with
$\ell \mathrel{{\mathcal Z}_{2  {\sf d}_X(k')-1}} s$. Moreover, $2  {\sf d}_X(\ell )+1 \leq 2  {\sf d}_X(k')-1$.
Ergo: $\ell\mathrel{{\mathcal Z}_{2  {\sf d}_X(\ell)+1}}s$. So $\ell,\ell,s$ is a witnessing triple for 
$\tupel{ \varGamma,s}$. Clearly, $\tupel{ \varDelta ,m}\sqsubset \tupel{ \varGamma,s}$. 
By the Induction Hypothesis:  
$\tupel{ \varGamma,n} \Vdash \psi$ and $\tupel{ \varGamma,n} \nVdash \psi$. Hence,
$ \tupel{ \varDelta ,m} \nVdash \psi \tto \chi$.

\medskip
\[
\begin{tikzcd}[column sep = large]
\varGamma   \arrow[mapsfrom,dashed]{r}{{\sf Th}_X} & \ell= \ell'  \arrow[-, dashed]{r}{{\mathcal Z}_{2{\sf d}_X(\ell)+1}}& s'=s \\
  \varDelta \arrow[mapsfrom]{r}{{\sf Th}_X}\arrow[dashed]{u}{\sqsubset, \neq}  & 
 k  \arrow[-]{r}{{\mathcal Z}_{2{\sf d}_X(k')}} \arrow[dashed]{u}{\sqsubset}& 
 m \arrow[dashed]{u}[swap]{\sqsubset} \\
 \varDelta \arrow[mapsfrom]{r}[swap]{{\sf Th}_X} \arrow{u}{=}& 
 k'  \arrow[-]{r}[swap]{{\mathcal Z}_{2{\sf d}_X(k'){+}1}}\arrow{u}{\preceq}
  &   m' \arrow{u}[swap]{\preceq}
\end{tikzcd} 
\]


\medent Conversely, suppose $\tupel{ \varDelta ,m} \nVdash \psi \tto \chi$. There is a 
$\tupel{ \varGamma ,n}$ in $\mathcal N$ with 
$\tupel{ \varDelta ,m}\sqsubset \tupel{ \varGamma ,n}$, $\tupel{ \varGamma ,n} \Vdash \psi $ and
 $\tupel{ \varGamma ,n} \nVdash \chi $. 
Clearly $\varDelta \sqsubset \varGamma$. By the Induction Hypothesis, $\varDelta \Vdash \psi$ and
$\varGamma  \nVdash \chi$. Ergo,
$\varDelta  \nVdash \psi \tto \chi$.  
\end{proof}
 
 \begin{remark}
 The definition of witnessing triple in the proof of Lemma~\ref{unijoin} is exactly the same
 as the definition given in \cite{viss:laye96} for the case of \ipc.  Thus, modulo a few details,
 our uniform interpolation proof is as simple as the proof of uniform interpolation for \ipc. 
 
 It is abundantly clear that something
 more complicated needs to be done for other intuitionistic modal logics than the ones we
 study here.
 \end{remark}
 
 \noindent
 We are now ready to prove Theorem~\ref{unii} for \sysunus.

\begin{proof}[Proof of Theorem~\ref{unii} for \sysunus]
 (1) We prove the existence of the post-interpolant $\exists \pvsb \,\phi$.
We consider $\phi$ and $\pvsb$. Let 
$\pvsa:={\sf PV}(\phi)\setminus \pvsb$. We take:
\[\exists  \pvsb\; \phi := \bigwedge \averz{\chi}{\cocl{2\nu +2}(\pvsa\,) }
{ \sysunus \vdash \phi\rightarrow \chi}.\]
Clearly, $\exists\pvsb\, \phi$ satisfies (a) and (c) of the statement of the theorem. Moreover, 
$\sysunus \vdash \phi \rightarrow \exists  \pvsb\; \phi$. Thus, all we have to prove is 
that for all $\psi$ with ${\sf PV}(\psi)\cap  \pvsb=\emptyset$, we have:
\[\sysunus\vdash  \phi  \rightarrow   \psi \;\;  \Rightarrow \;\;  
\sysunus\vdash  \exists  \pvsb \; \phi \rightarrow  \psi.\]

\noindent
Suppose, to obtain a contradiction, that, for some $\psi^\ast$ with
${\sf PV}(\psi^\ast)\cap  \pvsb=\emptyset$, we have 
$\sysunus\vdash  \phi \rightarrow  \psi^\ast$ and 
$\sysunus \nvdash  \exists  \pvsb \; \phi \rightarrow  \psi^\ast$.

We take $\pvsc:={\sf PV}(\psi^\ast)\setminus \pvsa$. 
Note that $\pvsa,\pvsb,\pvsc$ are pairwise disjoint, 
${\sf PV}(\phi)\subseteq \pvsb\cup\pvsa$ and
${\sf PV}(\psi^\ast)\subseteq \pvsa\cup\pvsc$.

Let $m$ be the root of a $\pvsa,\pvsc$-model $\mathcal M$ with 
$m\Vdash \exists  \pvsb \; \phi$ and $m \nVdash \psi^\ast$.
We define:
\begin{itemize}
\item 
$\theta^+ := \bigwedge \verz{\rho\in \cocl{2\nu+1}(\pvsa\,) \mid m\Vdash \rho }$
\item 
$\theta^{-} := \bigvee \verz{\sigma\in \cocl{2\nu+1}(\pvsa\,) \mid m \nVdash \sigma}$
\end{itemize}
We claim that $\phi,\theta^+ \nvdash_{\sysunus} \theta^-$. Suppose
$\phi\vdash_{\sysunus}  \theta^+\rightarrow \theta^-$. Hence, by definition, 
$\exists  \pvsb\, \phi  \vdash_{\sysunus}  \theta^+ \to \theta^-$. \emph{Quod non}, since 
$m\Vdash \exists  \pvsb \, \phi$ and $m\Vdash \theta^+$ and $m \nVdash \theta^-$. 

Let $k$ be the root of  some $\pvsb,\pvsa$-model $\mathcal K$ such that:
$k\Vdash \phi,\theta^+$ and $k \nVdash \theta^-$. We find that 
$k \simeq _{2\nu +1,\pvsa}m$. We note that $\nu = \nico{{\sf Sub}(\phi)}$.
We apply 
Lemma~\ref{unijoin}
with ${\sf Sub}(\phi)$ in the role of $X$ to find a  
$\pvsb,\pvsa,\pvsc$-node $n$ with:
$m \simeq_{\pvsa,\pvsc} n$ and 
${\sf Th}_{{\sf Sub}(\phi)}(k) = {\sf Th}_{{\sf Sub}(\phi)}(n)$.
It follows that  $n \nVdash \psi^\ast$, but $n\Vdash \phi$. A contradiction.

\bigskip\noindent
(2) We prove the existence of the pre-interpolant $\forall \pvsb\, \phi$.

Consider $\phi$ and $\pvsb$. Let $\pvsa:={\sf PV}(\phi)\setminus \pvsb$. 
We take
\[\forall \pvsb \; \phi := \bigvee \averz{\chi}{\cocl{2\nu +1}(\pvsa\,) }
{  \sysunus \vdash \chi\to  \phi}.\]
Clearly, $\forall \pvsb\, \phi$ satisfies (a) and (c). Moreover, 
$\sysunus \vdash \forall \pvsb \,\phi \rightarrow \phi$. Thus, all we
 have to prove is that, for all $\psi$ with 
${\sf PV}(\psi)\cap  \pvsb=\emptyset$, we have: 
\[\sysunus\vdash  \psi \rightarrow  \phi\;\; \Rightarrow\;\;  \sysunus\vdash  \psi \rightarrow  \forall \pvsb \, \phi.\]
Suppose that, to the contrary, for some $\psi^\ast$, we have
	${\sf PV}(\psi^\ast)\cap  \pvsb=\emptyset$   and 
$\sysunus\vdash  \psi^\ast \rightarrow  \phi$ and $\sysunus \nvdash  \psi^\ast \rightarrow  \forall \pvsb\, \phi$.
Take $\pvsc:={\sf PV}(\psi^\ast)\setminus \pvsa$. We note that 
$\pvsa$, $\pvsb$, and $\pvsc$ are pairwise disjoint and that
${\sf PV}(\psi^\ast)\subseteq \pvsb,\pvsa$ and ${\sf PV}(\phi)\subseteq \pvsa,\pvsc$.

Let $m$ be the root of a  $\pvsa,\pvsc$-model $\mathcal M$ with $m\Vdash \psi^\ast$ and 
$m \nVdash \forall \pvsb\; \phi$. Let 
\begin{itemize}
\item
$\theta^+:= \bigwedge\averz{\rho}{\cocl{2\nu+1}(\pvsa\,)}{m\Vdash \rho}$,
\item
$\theta^- :=  \bigvee\averz{\sigma}{\cocl{2\nu+1}(\pvsa\,)}{m\nVdash \sigma}$. 
\end{itemize}

\noindent
We claim that
$\theta^+ \nvdash_{\sysunus} \theta^-\vee \phi$. We note that, by Theorem~\ref{priyes}, the formula $\theta^+$ 
has the disjunction property. So, if 
$\theta^+\vdash_{\sysunus} \theta^-\vee \phi$, then $\theta^+\vdash_{\sysunus}  \theta^-$ or $ \theta^+\vdash_{\sysunus} \phi$.
Since $m$ witnesses that $ \theta^+ \nvdash_{\sysunus}  \theta^-$, it follows that $ \theta^+\vdash_{\sysunus}  \phi$. But, then, by
 definition, $ \theta^+\vdash_{\sysunus}  \forall \pvsb\,\phi$. \emph{Quod non}, since 
$m\Vdash_{\sysunus}   \theta^+$ and $m \nVdash_{\sysunus}  \forall \pvsb\,\phi$. 

Let $k$ be the root of a
$\pvsb,\pvsa$-model $\mathcal K$ such that $k\Vdash  \theta^+$ and 
$k \nVdash  \theta^-\vee \psi$. We find that 
$k \simeq_{2\nu +1,\pvsa}m$. We apply 
Lemma~\ref{unijoin}, with 
${\sf Sub}(\phi)$ in the role of $X$, to find a  
$\pvsb,\pvsa,\pvsc$-node $n$ with
$n \simeq_{\pvsa,\pvsc}m$ and 
${\sf Th}_{{\sf Sub}(\phi)}(k) = {\sf Th}_{{\sf Sub}(\phi)}(n)$.
It follows that $n\Vdash \psi^\ast$, but $n \nVdash \phi$. A contradiction.      
\end{proof}

\noindent

\subsection{Uniform Interpolation for \sysduo}\label{microsmurf}
We can prove uniform interpolation for \sysduo\ by an easy adaptation of
the argument for the case of \sysunus. The first step is to restate
Lemma~\ref{unijoin}.

\begin{lemma}\label{duojoin} Consider pairwise disjoint finite sets of propositional variables $\pvsb$, 
$\pvsa$ and $\pvsc$. Let $X\subseteq {\lang}(\pvsb,\pvsa\,)$
 be a finite adequate${}^+$ set. Let $\mathcal K$ be a brilliant $\pvsb,\pvsa$-model with root $k_0$, and let
$\mathcal M$ be a brilliant  $\pvsa, \pvsc$-model with root $m_0$. Let:
{\small
\[ \nu  := |\averz{\gamma}{X}{\text{$\gamma$ is a propositional variable, an implication or a lewis implication}}|. \]
}
Suppose that $k_0 \simeq _{2\nu  +1,\pvsa}m_0$. Then, 
there is a brilliant bisimulation $\pvsb$-extension $\mathcal N$  of $\mathcal M$ with root $n_0$ such that 
${\sf Th}_X(n_0) = {\sf Th}_X(k_0)$.      
\end{lemma}

\noindent
So, the only difference here is that we changed the class of models and the notion of adequacy. Then, in the construction of 
the model $\mathcal N$, we use $\mathbb H^{\circ+}$ instead of $\mathbb H^\circ$.
Finally, we note that the pairs construction of $\mathcal N$ does preserve brilliance.

We will give an alternative proof of uniform interpolation for \sysduo\ in Subsection~\ref{retrosmurf}.

\subsection{Bisimulation Quantifiers}
A nice aspect of uniform interpolation, in the cases under consideration,
 is that we can really view it as quantification on the semantical level.
 In the following theorem, we show that both for \sysunus\ and \sysduo\ we can view uniform interpolants
 as quantified formulas. We present the case of \sysduo\ between brackets.

\begin{theorem}[Semantical Characterisation of Pitts' Quantifiers]\label{sempi}
We consider \textup(brilliant\textup) $\widetilde a$-models for some finite set of propositional variables $\widetilde a$. 
Let $q$ be in $\widetilde a$ and let $\phi$ be a formula in $\lang(\widetilde a)$.
Let $m$ be a node of some \textup(brilliant\textup) $\widetilde a\setminus \verz q$-model $\mathcal M$. We have:
\begin{enumerate}[1.]
\item
$m\Vdash \exists\,  q\, \phi$ iff,
for some \textup(brilliant\textup) $\widetilde a$-model $\mathcal N$ with root  $n$, we have 
$m\simeq_{\widetilde a \setminus \verz{q}}n$ and $n\Vdash \phi$.
\item
$m\Vdash \forall q\,\phi$ iff, for all \textup(brilliant\textup) $\widetilde a$-models $\mathcal N$ and 
for all nodes $n$ in $\mathcal N$ with $m\simeq_{\widetilde a\setminus \verz{q}}n$, we have $ n\Vdash \phi$.
\end{enumerate}\end{theorem}

\begin{proof} (1) ``$\Leftarrow$'' Trivial. 

``$\Rightarrow$'' Let $\pvsa:={\sf PV}(\phi)\setminus \verz{q}$ and let $\pvsc := \widetilde a\setminus q,\pvsa$.
We take:
\[\nu  := |\verz{\psi\in {\sf Sub}(\phi) \mid \text{$\psi$ is an atom or an implication of a boxed formula}}|.\]
Suppose $m\Vdash \exists  q\, \phi$. We take:
\begin{itemize}
\item
$\theta^+ :=\bigwedge \verz{\psi \in \cocl{2\nu  +1}(\pvsa\,) \mid m \Vdash \psi}$
\item 
$\theta^- :=\bigvee \verz{\chi \in \cocl{2\nu  +1}(\pvsa\,) \mid m \nVdash \chi}$
\end{itemize}
As in the proof of Theorem~\ref{unii}(1), we have   $\phi, \theta^+ \nvdash \theta^-$. 
Let $\mathcal K$ be a model with root $k$ such that
$k\Vdash \phi,\theta^+$ and $k \nVdash \theta^-$. We find that 
$k \simeq_{2\nu +1,\pvsa}m$. We apply 
Lemma~\ref{unijoin} (Lemma~\ref{duojoin}) to $k$ and $m$ with ${\sf Sub}^{(+)}(\phi)$ in the role of $X$, $\verz{q}$ in 
the role of $\pvsb$, $\pvsa$ in the role of $\pvsa$, and $\pvsc$ in the role of 
$\pvsc$,  
to find a model $\mathcal N$ with root $n$ with:
$m\simeq_{\pvsa,\pvsc}n$ and 
${\sf Th}_{{\sf Sub}(\phi)}(k) = {\sf Th}_{\sf Sub(\phi)}(n)$, and, thus, $n\Vdash \phi$.

\bigskip\noindent
(2) The proof of (2) is similar. \end{proof}

\subsection{Uniform Interpolation via Retraction}\label{retrosmurf}

In \cite{viss:lobs05},  \visser\  studied interpretations in a somewhat more complicated context,
The following theorem from that paper also holds in our context. The concepts used in this subsection
were introduced in Subsection~\ref{smurfsmorf}.

\begin{theorem}\label{burkasmurf}
Suppose $\logvar'$ is a retract of $\logvar$ and we have uniform interpolation for $\logvar$.
Then, we have uniform interpolation for $\logvar'$.
\end{theorem}

\begin{proof}
Suppose $K:\logvar'\to \logvar$ and $M:\logvar\to \logvar'$ and $M\circ K = {\sf ID}_{\logvar'}$.
Let $\tau_K$ be a translation on which $K$ can be based and let $\tau_M$ be translation on which $M$ can be based. 

We treat the case of the post-interpolant. The case of the pre-interpolant is similar or, alternatively, can be reduced to that of the
pre-interpolant.
Let $\phi$ and $\psi$ be in the language of $\logvar'$.
Suppose  that the variables from $\pvsb$ do not occur in $\psi$. 
It follows that he variables from $\pvsb$ do not occur in $\psi^K$. 
We have:
\begin{eqnarray*}
\logvar'+\phi \vdash \psi & \To& \logvar + \phi^{\tau_K} \vdash \psi^{\tau_K} \\
& \To & \logvar+\exists \pvsb\; \phi^{\tau_K} \vdash \psi^{\tau_K} \\
& \To & \logvar' + (\exists \pvsb\; \phi^{\tau_K})^{\tau_M}\vdash \psi^{\tau_K\tau_M} \\
& \To & \logvar' + (\exists \pvsb\; \phi^{\tau_K})^{\tau_M}\vdash \psi
\end{eqnarray*}
We also have $\logvar + \phi^{\tau_K} \vdash \exists \pvsb\; \phi^{\tau_K}$.
So,
$\logvar' + \phi^{\tau_K\tau_M} \vdash (\exists \pvsb\; \phi^{\tau_K})^{\tau_M}$.
It follows that $\logvar' + \phi \vdash (\exists \pvsb\; \phi^{\tau_K})^{\tau_M}$.

We may conclude that $(\exists \pvsb\; \phi^{\tau_K})^{\tau_M}$ can be taken as the desired uniform post-interpolant for $\phi$ in
$\logvar'$.
\end{proof}

\noindent
We have the following corollary by combining Theorems~\ref{ochtendsmurf}, \ref{unii} for \sysunus\ and \ref{burkasmurf}.

\begin{corollary}\label{hobbysmurf}
We have uniform interpolation for \sysduo\ and \systres.
\end{corollary}

\noindent
Corollary~\ref{hobbysmurf} also gives us an estimate of the complexity of the uniform interpolants.
We first consider the case of $\sysduo$. Let us write:
\begin{itemize}
\item
$\mathfrak a(\phi) := |\verz{\pi \in {\sf Sub}(\phi) \mid \text{$\phi$ is an atom}}|$,\\
 $\mathfrak a^+(\phi) := |\verz{\pi \in {\sf Sub}^+(\phi) \mid \text{$\phi$ is an atom}}|$
\item
$\mathfrak i_\phi := |\verz{\psi\in {\sf Sub}(\phi) \mid \text{$\psi$ is an implication}}|$,\\
$\mathfrak i^+_\phi := |\verz{\psi\in {\sf Sub}^+(\phi) \mid \text{$\psi$ is an implication}}|$,
\item
$\mathfrak l_\phi := |\verz{\psi\in {\sf Sub}(\phi) \mid \text{$\psi$ is a Lewis implication}}|$,\\
$\mathfrak l^+_\phi := |\verz{\psi\in {\sf Sub}^+(\phi) \mid \text{$\psi$ is a Lewis implication}}|$,
\item
$\mathfrak n_{\phi}= \mathfrak a_{\phi}+\mathfrak i_{\phi} + \mathfrak l_{\phi}$, 
$\mathfrak n^+_{\phi}= \mathfrak a^+_{\phi}+\mathfrak i^+_{\phi} + \mathfrak l^+_{\phi}$.
\end{itemize}

\noindent 
We note that $\mathfrak a^+_{\phi} = \mathfrak a_{\phi}$, 
$\mathfrak i^+_{\phi} \leq \mathfrak i_{\phi}+ \mathfrak l_{\phi}$, $\mathfrak l^+_{\phi} = \mathfrak l_{\phi}$
and $\mathfrak n^+_{\phi} \leq \mathfrak n_{\phi} + \mathfrak l_{\phi}$.

The post-interpolant in \sysduo\  is $(\exists \pvsb\;\phi^{\sf triv})^{\sf id} = \exists \pvsb\;\phi^{\sf triv}$ and 
$\comp{\exists \pvsb\;\phi^{\sf triv}}\leq 2\cdot\mathfrak n_{\phi^{\sf triv}}+2$. 
It is easy to see that:
\[ \mathfrak n_{\phi^{\sf triv}}  =  \mathfrak n^+_{\phi} \leq  \mathfrak a_{\phi} + \mathfrak i_\phi + 2 \mathfrak l_{\phi}\]

\noindent We have similar considerations for the pre-interpolant. The interpretations provided by Theorem~\ref{ochtendsmurf}
witnessing  the retraction from \sysunus\ to \systres\ are $K :={\sf Triv}\circ {\sf BL}: \systres \to \sysunus$ and
$M := {\sf LB}\circ{\sf Emb}:\sysunus \to \systres$. The translation that supports $K$ is
${\sf triv}\circ {\sf lb} :=\ \opr \mapsto (\top \tto (\top \to p_0))$. It is immediate that the simpler translation ${\sf red} :=\opr \mapsto (\top \tto p_0)$
also supports $K$.  We  note that $ {\sf lb}\circ {\sf id} = {\sf lb}$ supports $M$.
The post-interpolant for \systres, can be taken $(\exists \pvsb\;\phi^{\sf red})^{\sf lb}$. It is easy to see that
the complexity for this interpolant is estimated by \[|\verz{\psi \in {\sf Sub}(\phi)\mid\text{$\phi$ is a variable or an implication or a boxed formula}}|.\]


\section{Arithmetical Interpretations}\label{arint}
In this section, we consider arithmetical interpretations of our systems \sysunus, \sysduo\ and \systres.
The section is purely expository. It surveys part of
\cite{viss:comp82}, \cite{viss:eval85}, \cite{viss:subs02}, \cite{iemh:moda01}, \cite{iemh:pres03}, \cite{iemh:prop05}, 
 \cite{lita:lewi18}, \cite{viss:prov19}, \cite{mojt:rela22}, \cite{mojt:prov22}.
 An excellent survey of part of the material is \cite{arte:prov04}.
Papers of related interest are \cite{dejo:embe96} and \cite{viss:inte98} and \cite{arde:sigm18}.

\subsection{Basics}
Consider an arithmetical theory\footnote{The restriction to arithmetical
theories is not really needed. We can formulate the framework for a much wider class of theories.
The present set-up, however, suffices for this paper.}
 $U$ and a language  $\lang(\arbop_0,\dots,\arbop_{n-1})$.
\emph{An arithmetical interpretation} $\mathcal I$ of $\lang(\arbop_0,\dots,\arbop_{n-1})$ is a mapping of each  $\arbop_i$ of arity $n_i$ to
an arithmetical formula in the variables $v_0,\ldots,v_{n_i-1}$. \emph{An arithmetical assignment} 
$\sigma$ is a mapping of
the propositional variables to arithmetical sentences.\footnote{Our terminology here is not the usual one. Sometimes `interpretation'   or `realisation'
is used, where we use `assignment'. But, then, mostly, what we call `interpretation' in the present context does not receive a name.}
We define $(A)^{\mathcal I,\sigma}$ by recursion as follows:
\begin{itemize}
\item
$(p_i)^{\mathcal I,\sigma} := \sigma(p_i)$,
\item
$(\cdot)^{\mathcal I,\sigma}$ commutes with the connectives of \ipc,
\item
$(\arbop_j(\phi_0,\dots,\phi_{n_j-1}))^{\mathcal I,\sigma} := 
\mathcal I(\arbop_j)(\gn{(\phi_0)^{\mathcal I,\sigma}}, \dots, \gn{(\phi_{n_j-1})^{\mathcal I,\sigma}})$.
\end{itemize}
We define $\logcon_U^{\mathcal I}$ as the set of all $\phi$ such that, for all $\sigma$, we have $U \vdash (\phi)^{\mathcal I,\sigma}$.
We note that $\logcon_U^{\mathcal I}$ is always closed under substitution and Modus Ponens, but at this level of generality little more of interest can be said.

The provability logic of $U$ is $\logcon_U^{\mathfrak I_{0,U}}$ for the language with the single modal operator $\opr$, where
$\mathfrak I_{0,U}(\opr) = {\sf prov}_U(v_0)$. Here {\sf prov} is a standard arithmetisation of the provability predicate.\footnote{We 
are assuming that the set of G\"odel numbers of the axioms of $U$ is given by a designated arithmetical formula $\alpha$.
When we write ${\sf prov}_U$, we really mean ${\sf prov}_\alpha$. In this paper, we will assume that $\alpha$ is sufficiently simple, specifically that it is
$\Sigma_1^{\sf b}$.} 
As soon as $U$ contains a suitable constructive version of Buss' theory ${\sf S}^1_2$,  we find that the
provability logic of $U$ contains $\ik+\lsb{\Lo}$, the constructive version of L\"ob's Logic.\footnote{Inspection of the
verification of the L\"ob Conditions in classical ${\sf S}^1_2$ reveals that this verification is fully constructive.
See \cite{buss:boun86} for the main ingredients of the proof.} We note that this presupposes that the representation of the axiom set
is sufficiently simple.

In the definition of provability logic, the parameter $U$ occurs both as base theory and in the determination of the interpretation of
$\opr$. This is the reason that the provability logic of a theory need not be monotonic in that theory. As soon as we restrict ourselves
to extensions of classical Elementary Arithmetic, {\sf EA}, the mapping $U \mapsto \logcon_U^{\mathfrak I_{0,U}}$ becomes monotonic.
However, this surprising fact rests on the very specific properties of \emph{classical} provability logic.

We will focus on \emph{$\Sigma_1$-preservativity} as arithmetical interpretation of $\tto$. This notion is defined as follows.
\begin{itemize}
\item
$A \tto_{U,\Sigma^0_1} B$ iff, for all $\Sigma_1^0$-sentences $S$,  if $U \vdash S\to A$, then $U\vdash S \to B$. 
\end{itemize}
For a discussion of some other interpretations, see \cite{lita:lewi19}.
We take $\mathfrak I_{1,U}(\tto) := {\sf pres}_{U,\Sigma^0_1}(v_0,v_1)$, where
${\sf pres}_{U,\Sigma^0_1}(v_0,v_1)$ is some standard arithmetisation of
$\Sigma^0_1$-pre\-ser\-va\-tivity over $U$. 

For extensions $U$ of the intuitionistic version {\sf i-EA} of Elementary Arithmetic, {\sf EA}, the logic
$\logcon_U^{\mathfrak I_{1,U}}$ extends \iam, i.e., \ia\ minus \di,  plus \lsa{L} (and, hence, \lsa{4}).
Moreover, $\logcon_U^{\mathfrak I_{1,U}}$ satisfies Montagna's Principle:
\begin{description}
\item[\mont]
$(\phi \tto \psi) \to ((\opr\chi \to \phi)\tto(\opr\chi \to \psi))$
\end{description}
We do, however, not generally have \di, as is witnessed by the case of {\sf PA}.
If $U$ is, e.g., {\sf HA} or ${\sf HA}^\ast$ (see below), we do have \di.\footnote{This
follows from the fact that these theories are closed under q-realisability. See \cite{lita:lewi18} for an explanation.}
The relation of 
$\Sigma^0_1$-preservativity is technically a very useful notion in studying the provability
logic of {\sf HA}. Its logic for interpretations in {\sf HA} is currently unknown.
 
 \subsection{The Provability Logic of Heyting Arithmetic}
The provability logic of {\sf HA} contains many extra principles over $\ik+\lsb{\Lo}$. Here are some examples.
\begin{itemize}
\item
$\opr(\psi \vee \chi) \to \opr(\psi\vee \opr\chi)$ \hspace*{1cm} (Leivant's Principle)
\item
$\opr \neg\neg\,\opr\phi \to \opr\opr\phi$ \hspace*{1.86cm} (Formalised Markov's Rule)
\item
$\opr (\neg\neg\,\opr\phi \to \opr\phi) \to \opr\opr\phi$ \hspace*{0.64cm} (Formalised Anti-Markov's Rule)
\end{itemize}
There are many more examples of principles. 
We note that in $\ik+\lsb{\Lo}+{\sf Le}$, where {\sf Le} is Leivant's Principle, we can derive
\begin{eqnarray*}
\vdash \opr(\opr\bot \vee \neg\, \opr\bot) & \To & \vdash \opr(\opr\bot \vee \opr \neg\, \opr\bot) \\
& \To & \vdash \opr(\opr\bot \vee \opr\bot) \\
& \To &  \vdash \opr\opr\bot
\end{eqnarray*}

This tells us that over classical provability logic {\sf GL}, Leivant's Principle is equivalent to $\opr\opr\bot$.
The Formalised Anti-Markov's Rule is similarly equivalent to $\opr\opr\bot$ even over the classical theory {\sf K}.

These two examples illustrate non-monotonicity: {\sf PA} extends {\sf HA} and both Leivant's Principle and the 
Formalised Anti-Markov's Rule  are in 
$\logcon_{\sf HA}^{\mathfrak I_{0,{\sf HA}}}$, but not in $\logcon_{\sf PA}^{\mathfrak I_{0,{\sf PA}}}$.

Here are some results on the provability logic of {\sf HA}.
\begin{itemize}
\item
Rosalie Iemhoff gave a characterisation of the Admissible Rules of \ipc. See \cite{iemh:admi01}.
In combination with the results from \cite{viss:rule99}, this also gives us a characterisation of the
verifiably admissible rules of Heyting Arithmetic. Thus, we find a characterization of all
the formulas of the form $\opr\phi \to \opr\psi$, where $\phi$ and $\psi$ are purely propositional,
that are in $\logcon_{\sf HA}^{\mathfrak I_{0,{\sf HA}}}$.
\item
Albert Visser gave a characterisation of the closed fragment of the provability logic of {\sf HA}. 
See \cite{viss:subs02}. For another approach to the same problem, see also \cite{viss:close08}.
\item
Mohammad Ardeshir and Mojtaba Mojtahedi characterised the provability logic of $\Sigma^0_1$-substitutions
in {\sf HA}. See their article \cite{arde:sigm18}. 
\item
Mojtaba Mojtahedi claims a characterisation of the provability logic of {\sf HA}. This characterisation implies its
decidability. See \cite{mojt:rela22} and \cite{mojt:prov22}. The result and its proof look very plausible, but at the time
of writing there is no full check of the proof by a third party.
\end{itemize}

\subsection{The Completeness Principle}
Heyting Arithmetic {\sf HA} has a unique extension ${\sf HA}^\ast$ such that,
{\sf HA}-verifiably, we have:
${\sf HA}^\ast = {\sf HA} + (A \to \opr_{\,{\sf HA}^\ast} A)$. 
This theory was introduced in \cite{viss:comp82}. 
We will call the scheme $(A \to \opr_T A)$ the Completeness Principle for $T$ or ${\sf CP}_T$.

The theory ${\sf HA}^\ast$ is, perhaps, not philosophically a convincing theory. However,
it is rich in applications as an auxiliary tool. 
{\small
\begin{itemize}
\item
Michael Beeson  proved the independence of
the principles {\sf KLS} and {\sf MS}, concerning the continuity of the effective operations,
 over {\sf HA} in his paper \cite{bees:nond75} using the method of fp-realisability. Albert Visser simplified the proof using
 ${\sf HA}^\ast$ in \cite{viss:comp82}. 
 \item
 ${\sf HA}^\ast$ can be employed to give a proof of de Jongh's Theorem for {\sf HA} using $\Sigma^0_1$-substitutions. See \cite{viss:subs02}.
 \item
 ${\sf HA}^\ast$ can be used to show the schematic optimality of the theorem that ${\sf HA}$ proves that, if a Boolean combination of
 $\Sigma^0_1$-sentences is {\sf HA}-provable, then its best approximation in the {\sf NNIL}-formulas is also
 {\sf HA}-provable. See \cite{viss:subs02}.
 \item
${\sf HA}^\ast$ can be used to show that the Heyting algebra of {\sf HA} has a subalgebra on three generators that is complete recursively
enumerable.
 See \cite{dejo:embe96}.
 \item
 Mohammad Ardeshir and Mojtaba Mojahedi \cite{arde:sigm18} proved a 
 characterisation of the provability logic of {\sf HA} for  $\Sigma^0_1$-substitutions.
 Jetze Zoethout  and  Albert Visser \cite{viss:prov19} showed how to reprove the result using the $(\cdot)^\ast$-construction of which
 ${\sf HA} \mapsto {\sf HA}^\ast$ is a special case, as a tool.
See also the brief sketch below.
 \end{itemize}
 }
 
 \noindent
 The theory ${\sf HA}^\ast$ is the unique solution of the equation ${\sf HA}^\ast =_{\sf HA} {\sf HA} + {\sf CP}_{{\sf HA}^\ast}$,
 where $=_{\sf HA}$ indicates that {\sf HA} verifies the identity. We can also obtain ${\sf HA}^\ast$ by the following construction.
 Consider an arithmetical theory $U$ that extends i-{\sf EA}. We assume $U$ is given with a suitable elementary presentation of its axioms.
 We define a translation $(\cdot)^U$ from the arithmetical language to itself as follows: the translation commutes with atoms,
 conjunction, disjunction and the existential quantifier. Moreover,
 \begin{itemize}
 \item
  $(A\to B)^U := (A^U \to B^U) \wedge \opr_U(A^U\to B^U)$,
  \item
   $(\forall x \, A)^U := \forall x\, A^U \wedge \opr_U \forall x\, A^U$.
   \end{itemize}
  We define $U^\ast$ to be the theory given by $\verz{A \mid U \vdash A^U}$. One can verify that
  $U^\ast \vdash {\sf CP}_{U^\ast}$ and that $U^\ast$ is i-{\sf EA}-verifiable a subtheory of $U+{\sf CP}^\ast_{U^\ast}$.
  In the case of {\sf HA} we also find that ${\sf HA}^\ast$ contains {\sf HA}.

It is immediate that this logic of  contains  $\systres$. One might hope that the provability logic of ${\sf HA}^\ast$ is precisely \systres.
However, as was pointed out to  \visser\  by Mojtaba Mojtahedi, the provability logic
of  ${\sf HA}^\ast$ strictly extends $\systres$. 
Here is an example:
\[  \opr(\opr\bot \to (\neg \, \phi \to (\psi\vee \chi))) \to \opr(\opr\bot \to ((\neg\, \phi \to \psi) \vee (\neg\,\phi \to \chi))).\footnote{This example
reflects the fact that ${\sf HA}+{\sf incon}({\sf HA})$ extends ${\sf HA}^\ast$ and is closed under the Friedman translation.}  \] 

\noindent
Thus, the provability logic of ${\sf HA}^\ast$ reflects some of the complexity that we find in the provability logic of
{\sf HA}. Here are a few quick insights. 
First, the closed fragment of the provability logic of ${\sf HA}^\ast$ can be easily seen to be the closed fragment of \systres.
Secondly, de Jongh and  \visser\  in \cite{dejo:embe96} show that the admissible rules for assignments of propositional logic
in ${\sf HA}^\ast$ are precisely the derivable rules of {\sf IPC}. So, we have, for propositional formulas $\phi$ and $\psi$:
\begin{eqnarray*}
\logcon_{{\sf HA}^\ast}^{\mathfrak I_{0,{\sf HA}^\ast}} \vdash \opr\phi \to \opr\psi & \To &
 \forall \sigma \;{\sf HA}^\ast \vdash {\sf prov}_{{\sf HA}^\ast}(\gn{\phi^\sigma})
 \to {\sf prov}_{{\sf HA}^\ast}(\gn{\psi^\sigma}) \\
 & \To &  \forall \sigma \; ({\sf HA}^\ast\vdash \phi^\sigma\; \To\; {\sf HA}^\ast\vdash \psi^\sigma)  \\
 & \To & \ipc \vdash \phi \to \psi \\
 &\To& \logcon_{{\sf HA}^\ast}^{\mathfrak I_{0,{\sf HA}^\ast}} \vdash \opr\phi \to \opr\psi
\end{eqnarray*}

\noindent
The second step uses the $\Pi^0_2$-soundness of ${\sf HA}^\ast$ and the third step the de Jongh-Visser result.

We can use the $(\cdot)^\ast$-construction fruitfully by shifting to a different line of questioning.
Instead of a question of the form
\begin{description}
\item[A]
 \emph{What is the provability logic of a given theory $U$?}
 \end{description}
we invert the direction of our gaze, and ask a question of the form:
\begin{description}
\item[B]
 \emph{Suppose we are given a modal logic
$\logvar$.  Can one find an arithmetical theory $U$ such that the  provability logic
of $U$ is precisely $\logvar$?}
\end{description}
More precisely,  question (B) should be asked about presentations of the axiom set of $U$, since, generally, the question of the
provability logic of a theory is intensional.

 In the classical context, some of the work of Taishi Kurahashi is precisely about questions of
 form (B). See \cite{kura:arit18a} and \cite{kura:arit18b}.

We may ask question (B), for the case where $\logvar$ is \systres.  We have already seen that
${\sf HA}^\ast$ is not the desired answer. However, there is a hint that suggests where
the answer could lie. In the classical context, arithmetical completeness theorems are usually
proved using Solovay's method. This method, involves, roughly, `embedding' Kripke models in arithmetic. 
The first author, in his PhD Thesis \cite{viss:aspe81}  showed that  Kripke models 
for \systres, in the case $\preceq$ and $\sqsubseteq$ coincide, can be embedded Solovay-style in ${\sf HA}^\ast$.
Given that the logic of this restricted class of models is precisely {\sf KM}, a system introduced by
 Kutnetzov and Muravitsky (see \cite{KuznetsovM86:sl}), it follows that {\sf KM} is an upper bound for the provability logic
of ${\sf HA}^\ast$. Since,  $ {\sf KM}+ \opr\bot$ implies classical propositional logic and ${\sf HA}^\ast+ \opr_{{\sf HA}^\ast}\bot$
does not, {\sf KM} is not sound for provability interpretations in ${\sf HA}^\ast$.\footnote{In fact the propositional logic of
{\sf HA} plus any independent $\Sigma_1$-sentence is still Intuitionistic Propositional Logic.
For example, Smory\'nski's proof of de Jongh's Theorem for {\sf HA} in \cite[5.6.13--5.6.16]{smor:appl73}, works as well
for such extensions. Many other proofs of de Jongh's Theorem for {\sf HA} can be similarly adapted. It is an open question
whether any consistent finite extension of {\sf HA} still satisfies de Jongh's Theorem.} 
 So, how proceed for the cases where $\preceq$ and $\sqsubseteq$ are different?
What one needs is a second provability notion to handle $\preceq$. Some reflection suggests that
we need ${\sf S}_{\opr}$  for this second, weaker kind of provability. It turned out that this line of reasoning leads to
the answer. The solution was given by
Jetze Zoethout and  \visser\  in \cite{viss:prov19}. In the remainder of this subsection, we sketch the essence of that solution. 

The second provability relation, employed by Zoethout and \visser, is \emph{slow provability}. This notion was introduced by
Sy Friedman, Michael Rathjen and Andreas Weiermann in their paper \cite{frie:slow13}.\footnote{A somewhat similar
idea, in the context of Elementary Arithmetic, had already been considered in \cite{viss:seco12}.}

Let ${\sf F}_\alpha$ be the fast 
growing hierarchy due to Stan Wainer \cite{wain:clas70} and Helmuth Schwichtenberg \cite{schw:klas71}. 
We say that a number $x$ is \emph{small} when ${\sf F}_{\epsilon_0}(x)$ exists. Of course, 
\emph{we} know that all numbers all small, but {\sf HA} does not know it.
Let ${\sf HA}^{\sf s}$ be {\sf HA} with its axioms restricted to the axioms with small G\"odel numbers.   
We call provability in ${\sf HA}^{\sf s}$: \emph{slow provability}.\footnote{We think this is a somewhat
unhappy terminology, since it suggests that the proofs rather than just the axioms are small. However,
such is the legacy.}
We note that slow provability is a Fefermanian notion, since, in accordance with
Feferman's program in \cite{fefe:arit60}, we keep the arithmetisation of provability fixed and only tamper
with the representation of the axiom set.\footnote{A different approach to slow provability, following a suggestion by Fedor Pakhomov,
 is given in \cite{viss:abso20}. See also \cite{sier:prov23}. 
It has the advantage that the relevant properties of slow provability are easily verifiable in a wide range of theories.}

We consider ${\sf HA}^{{\sf s}\ast}$.
The theory ${\sf HA}^{{\sf s}\ast}$ need not be an extension of ${\sf HA}^{\sf s}$.
However, we do have that, {\sf HA}-verifiably,  the theory ${\sf HA}^{{\sf s}\ast}$ is a subtheory of ${\sf HA}^{\sf s}+{\sf CP}_{{\sf HA}^{{\sf s}\ast}}$.
We note what the instances of  ${\sf CP}_{{\sf HA}^{{\sf s}\ast}}$ here are not constrained to being small.
Finally, we extend  ${\sf HA}^{{\sf s}\ast}$ to, say, $\widetilde {\sf HA}$ by adding the ordinary axiom set of {\sf HA}.
We can show that $\widetilde {\sf HA}$ is, {\sf HA}-verifiably, ${\sf HA}+{\sf CP}_{{\sf HA}^{{\sf s}\ast}}$.
So,  $\widetilde {\sf HA}$ is a non-slow theory that proves the completeness principle of a somewhat slower sibling.

Finally, ${\sf HA}^{{\sf s}\ast}$-provability is used for our weaker notion of provability and
$\widetilde{\sf HA}$-provability is the stonger one.

Using the ideas sketched above, we find that $\systres$ is precisely the provability logic of
$\widetilde {\sf HA}$. 

\begin{question}
Does the $\Sigma^0_1$-preservativity logic of  $\widetilde {\sf HA}$ contain any principles over and above
those of \sysunus? If not can we find a suitable sibling of $\widetilde {\sf HA}$ that yields the desired completeness
result?

A hopeful idea would be that the substitutions considered by Zoethout and  \visser\  are already sufficient to give the
desired completeness result. However, the arithmetical interpretations of modal formulas considered by Zoethout and  \visser\ 
are all $\widetilde {\sf HA}$-provably equivalent to $\Sigma^0_1$-sentences. It follows that these interpretations cannot
distinguish between $\phi \tto \psi$ and $\opr(\phi\to\psi)$. So, something new is needed.\footnote{The usual example to separate
provable implication from $\Sigma^0_1$-preservativity is $\Sigma^0_2$.}
\end{question}


\noindent The Zoethout-Visser result leads to an alternative proof of the characterisation of the provability logic of
{\sf HA} for $\Sigma^0_1$-substitutions, a result that is originally due to Mohammad Ardeshir {\&} Mojtaba Mojtahedi \cite{arde:sigm18}.


\section{\systres\ in Computer Science, Type Theory and Categorical Logic} \label{sec:nakano}

 \newcommand{\interm}{\mbox{\textcolor{uuxred}{$\opr$}}}

The work of Nakano on the \emph{modality for recursion} \cite{Nakano00:lics,Nakano01:tacs} 
has brought the axioms of  \systres\ to the attention of researchers in type theory and 
computer science.  Related modalities have subsequently appeared as, e.g., 
\begin{itemize}
 \item 
 the  \emph{guardedness type constructor} \cite{AtkeyMB13:icfp,AbelV14:aplas}, 
\item 
 the  \emph{approximation modality}  \cite{AppelMRV07:popl},
\item
  the \emph{next-step modality/next clock tick} \cite{KrishnaswamiB11:icfp,KrishnaswamiB11:lics} 
\item
 or perhaps most commonly the \emph{later operator} \cite{BentonT09:tldi,BirkedalM13:lics,BirkedalMSS12:lmcs,JaberTS12:lics,BizjakEA16:csl,CloustonEA17:lmcs,MogelbergP19:mscs,BirkedalEA19:JAR}.
\end{itemize}

\noindent This incomplete list of references  illustrates the post-Nakano success of L\"ob modalities in CS-oriented type theories. Few of them follow Nakano's exact subtyping setup (presented in detail in Appendix \ref{nakorig}). The motivation varies from providing a \emph{lightweight typing discipline for enabling productive coprogramming} \cite{AtkeyMB13:icfp} to solving ``circular'' equations arising in semantical analysis of programming languages, e.g., in proofs of type safety, 
 contextual approximation and equivalence of programs  (via \emph{binary} logical relations \cite{DreyerAB11:lmcs}). In theoretical computer science, such semantical analysis is often conducted in category-theoretic terms.  
  Birkedal et al. \cite{BirkedalMSS12:lmcs} proposed the \emph{topos of trees} (that is, the topos of presheaves on $\omega$; see below for a detailed discussion) as a natural model of ``guarded recursive definitions of both recursive functions and relations as well as recursive types''. Furthermore, they add
 \begin{quote}
\dots the external notion provides for a simple algebraic theory of fixed points
for not only morphisms but also functors (see Section 2.6), whereas the internal notion is useful when working in the internal logic \cite{BirkedalMSS12:lmcs}.
\end{quote}
In order to parse these statements, let us recall from standard references \cite{LambekS86:ihocl,Pitts00:catl} that 
 there are several ways in which  categories with sufficiently rich structure encode constructive logics, in particular the intuitionistic calculus.
This stems from a more fundamental fact: propositions can be interpreted in terms of  types (\S\ \ref{sec:proptypes}), in terms of predicates/properties  (\S\ \ref{sec:proppred}) or even in terms of truth values  (\S\ \ref{sec:propomega}). A lot can be said about the status and importance of (categorical) strong L\"ob modalities in each of  these cases.  As a first approximation: the distinction between the use of L\"ob-style modalities in these contexts resembles the distinction between the use of the lax (\lna{PLL}) modality in the Curry-Howard correspondent of Moggi's computational lambda calculus  \cite{BentonBP98:jfp} and its use as ``geometric modality'' to generalize Grothendieck topology to arbitrary toposes \cite{Goldblatt81:mlq}.

\subsection{$\systres$ Endofunctors: Guarded Fixpoint Categories} \label{sec:proptypes}
 The Curry-Howard-Lambek interpretation of \ipc\ in (bi)cartesian closed categories maps propositional connectives directly to operations on objects (which interpret propositional formulas): conjunction to product, disjunction to coproduct, and implication to formation of the exponential object (function space). 
Under the CS view of this interpretation, objects of the category are semantic counterparts of types of a suitable programming language and morphisms between them (which logically correspond to proofs) are  semantic counterparts of programs.

If propositions correspond to types, and these correspond to objects in a given (bi)cartesian (closed) category, then modalities obviously correspond to endofunctors on the category in question \cite{BiermanP00:sl,BellinPR01:m4m,dePaivaR11}. In particular, Milius and Litak \cite{MiliusL17:fi} characterize what one might call strong L\"ob endofunctors\footnote{The paper has been written without employing explicit proof-theoretic terminology and hiding logical aspects and inspiration.} in terms of \emph{guarded fixpoint operators}. First of all, the category $\catC$ in question needs to model  basic type-theoretic (propositional) constructs; at the very least, it needs to be cartesian (equipped with products), though all models of practical interest  are cartesian closed (equipped with exponentials). Furthermore, $\catC$ needs to be a categorical model of $\C$. That is, we need to specify a \emph{pointed} (\emph{delay}) \emph{endofunctor} $\ibox: \catC \to
  \catC$, i.e., one equipped with a natural transformation $\point: \Id \to
  \ibox$. 
The reader is asked to note here that there are several different categorical modalities in this section ($\ibox$,  $\pbmod{\gamma}{X}$ and $\interm$), between which we are making subtle notational distinctions.

\begin{definition}
  \label{def:dagger}
  A \emph{guarded fixpoint operator} on a cartesian category $\catC$ is a family of
  operations
  \[
  \dagger_{X,Y} : \catC(X \times \ibox Y, Y) \to \catC(X,Y)
  \]
  such that for every $f: X \times \ibox Y \to Y$ the following square
  commutes\footnote{Notice that we use the convention 
      of simply writing objects to denote the identity morphisms on them.}:
    \begin{equation}\label{eq:fixp}
      \vcenter{
        \xymatrix@C+1pc{
          X 
          \ar[r]^-{\sol f}
          \ar[d]_{\langle X, \sol f\rangle }
          &
          Y
          \\
          X \times Y
          \ar[r]_-{X \times \point_Y}
          &
          X \times \ibox Y
          \ar[u]_{f}
        }
      }
    \end{equation}
    where we drop the subscripts and write $\sol f:X \to Y$ in lieu of
    $\dagger_{X,Y}(f)$. A \emph{guarded fixpoint category} is a triple $(\catC,\ibox,\dagger)$.
\end{definition}

This is what one needs to categorically model a proof system for $\systres$,  in particular the rule
\inferrule{\varGamma,\opr\phi \vdash \phi}{\varGamma \vdash \phi}. Furthermore, Milius and Litak  \cite{MiliusL17:fi} study equational laws for morphisms corresponding to natural rules for proof rewriting and normalization. The reader is encouraged to compare this with standard  references on Curry-Howard-Lambek interpretation of modal natural deduction systems for better-known calculi \cite{BiermanP00:sl,BellinPR01:m4m,dePaivaR11}.

Note that as a corollary we obtain that if $X = \top$ and $g: Y \to Y$ factors through $\point_Y$ via $f: \ibox Y \to Y$ (such $g$ are called \emph{contractive}), then the global element $\sol f: \top \to Y$ is a fixpoint of $g$. Note furthermore that $\sol f$ in question may, but does not have to be uniquely determined; if it is, we speak of a \emph{unique guarded fixpoint operator} and, respectively, \emph{unique guarded fixpoint category}. The definition of a \emph{weak model of guarded fixpoint terms} by Birkedal et al. \cite[Def. 6.1]{BirkedalMSS12:lmcs} in fact assumes precisely the existence of unique fixpoints for $X = \top$ (from the proof-theoretic point of view, this means that the context $\varGamma$  is empty) for pointed (\C) endofuctors which preserve products (correspond to normal modalities). As shown by Milius and Litak \cite[Prop. 2.7]{MiliusL17:fi}, these 
conditions yield a guarded fixpoint operator in the above sense whenever the underlying cartesian category in question is in fact cartesian closed (i.e., it models Heyting implication as an operator on types). Seen from the perspective of modal proof theory, this observation is rather unsurprising.

Milius and Litak \cite[Ex. 2.4]{MiliusL17:fi} list several different classes of examples:
\begin{itemize}
\item An important if modally degenerate class of guarded fixpoint operators is provided by categories with an ordinary fixpoint operator \cite{h97,h99,sp00} (where $\ibox$ is taken to be the identity functor).
\item For any cartesian category, one can take $\ibox$ to be just the constant functor sending all objects to the terminal element.
\item One can define a guarded dagger structure on any \emph{let-ccc with a fixpoint object}  \cite{cp92,Crole:phd}, although only in the special case of the lifting functor $(\cdot)_\bot$ on $\omega$-complete partial orders do all the (guarded) Conway axioms including the \emph{double dagger identity} \cite[Definition 3.1]{MiliusL17:fi} hold. In this case, there is no hope for uniqueness. 
\item Another class \cite[\S\ 2.2]{MiliusL17:fi}  of (unique!) guarded fixpoint operators is obtained via duals of Kleisli categories of \emph{completely iterative monads} \cite{m05}, which provide a monadic presentation of Elgot's (\emph{completely}) \emph{iterative
theories}~\cite{elgot75,ebt78}.
\item In the category \lna{CMS} of complete 1-bounded\footnote{A  metric space is 1-bounded iff, for any $x,y \in
X$, we have ${\sf d}_X(x,y) \leq 1$} metric spaces and
non-expansive maps, i.e.\ maps $f: X \to Y$ such that for all $x,y \in
X$, we have ${\sf d}_Y(fx,fy) \leq {\sf d}_X(x,y)$ 
(cf., e.g.,~\cite{KrishnaswamiB11:lics,KrishnaswamiB11:icfp} or ~\cite[\S\ 5]{BirkedalMSS12:lmcs} and references therein), we can fix an arbitrary $r \in (0,1)$ and then take $\ibox$ to be the endofunctor keeping the carrier of the space and multiplying all distances by $r$. The proof that this yields a (unique!) guarded fixpoint operator relies directly on Banach's fixpoint theorem. It should be noted that spaces typically arising in CS applications (including the work of Nakano, cf. the last sentence of Appendix \ref{nakorig}) are not only metric, but ultrametric and moreover \emph{bisected}: all distances are negative multiples of 2. In such a setup,  one can canonically choose $r$ to be $0.5$. 
\item For any Noetherian poset $(W, <)$---i.e., one with no infinite descending chains---and any cartesian category $\mathcal{A}$, we consider the functor category $\mathcal{A}^{(W,>)}$, which can be also called the category of contravariant $(W,<)$-presheaves in $\mathcal{A}$. If either $(W,<)$ is the just order type $\omega$ of natural numbers or  if $\mathcal{A}$ is \emph{complete} (having the limits of all diagrams), then the resulting category allows  a (unique!) guarded fixpoint operator. For $W = \omega$, The delay functor 
      is given by
      $\ibox X(0) \deq 1$ and $\ibox X(n+1) \deq X(n)$ for $n \geq 0$,
      whereas $\point_X$ is given by
      {\small
       \[(p_X)_0 \deq \mathord{!} : X(0) \to 1 \text{ and }(\point_X)_{n+1}
      \deq X (n+1 \geq n): X(n+1) \to X(n).\]
      }
       For
      every $f: X \times  \ibox Y \to Y$ there is a unique $\sol f: X
      \to Y$ satisfying~\refeq{eq:fixp} given by 
      $\sol f_0 \deq f_0: X(0) \to Y(0)$ and 
      {\small
      \[
      \sol f_{n+1}
      \deq
      (\xymatrix@1{
        X(n+1) \ar[rrrr]^-{\langle X(n+1),\sol f_n \cdot X(n+1 \geq n)
          \rangle}
        &&&&
         X(n+1) \times Y(n)
        \ar[r]^-{f_{n+1}}
        &
        Y(n+1)
      }).\]
      }
\end{itemize}

Perhaps the most important concrete instance of the last example is the above-mentioned topos of trees \cite{BirkedalMSS12:lmcs}, i.e., $\mathsf{Set}$-valued presheaves on $\omega$. As noted by Birkedal et al. \cite[\S\ 5]{BirkedalMSS12:lmcs}, the subcategory of this topos consisting of total objects is  equivalent to that of bisected, complete ultrametric spaces. But in the topos setting, logical propositions can be read not only as types, but also  as predicates or properties. And in fact, this reading is available in broader classes of categories. 

\subsection{$\systres$ and Predicate Transformers: Well-founded Delay Coalgebras} \label{sec:proppred}

A categorical account of the notion of a property over a given object $C \in \ctE$, somewhat more flexible than the one provided by \emph{hyperdoctrines} but following the same approach, is given by Pitts \cite[Def. 5.2.1]{Pitts00:catl}.  
\begin{definition}
A pair $(\ctC,\Prop)$ consisting of a cartesian  category $\ctC$ and a contravariant functor $\Prop$ from $\ctC$ to the category of poset and order-preserving mappings is called a \emph{prop-category}. For any object $X \in \ctC$, $\Prop(X)$ is the collection of ($\ctC$-)\emph{properties} of $X$.
\end{definition}

In many categories, the poset of (equivalence classes of) \emph{subobjects} of $X$ provides a natural candidate for the collection of its properties (although Pitts \cite{Pitts00:catl} illustrates that in some cases there might be better choices). Recall that a \emph{subobject} of $X$ is of the form $A \stackrel{m}{\monar} X$, with $m$ being a \emph{monomorphism} (also known as \emph{monic} or \emph{mono}), i.e., left-cancellative.  This is an obvious generalization of (an isomorphic copy of) a subset of $X$. We say that $m \leq m'$ whenever $m$ factors through $m'$ and identify $\leq$-equivalence classes. We say that a category is \emph{well-powered} if for any $X$ the collection of such equivalence classes  is a set rather than a proper class and write $\PropS{X}$ for the resulting poset. 

For any $X$, the poset $\PropS{X}$ obviously has a greatest element, namely the identity on $X$. If $\ctC$ has an initial object, $\PropS{X}$ always has a least element.  If $\ctC$ has pullbacks, $\PropS{X}$ is in fact a meet-semilattice, and if $\ctC$ has arbitrary limits, a complete lattice. 


  Under this perspective,  type-level modalities-as-endofunctors  induce modalities as \emph{predicate transformers} acting on semilattices of properties. More specifically, whenever $\ibox: \ctE \to \ctE$ is   monic-preserving and $\ctC$ has pullbacks, associate with a $\ibox$-coalgebra $X \stackrel{\gamma}{\to} \ibox X$ a modality $\pbmod{\gamma}{X}$ on $\PropS{X}$: 

\begin{equation} \label{eq:pbcl}
\vcenter{
\xymatrix{
\pbmod{\gamma}{X} M \ar[d]\ar@{>->}[r]^-{\pbmod{\gamma}{X} m}_<<{\!\!\!\!\!\!\!\pb} & X \ar[d]^{\gamma}\\
\ibox M \ar@{>->}[r]^{\ibox m    } & \ibox\, X
}}
\end{equation}

\noindent (see \cite{AdamekMMS12:fossacs} for the history of this diagram in papers on well-founded coalgebras starting with Taylor). Furthermore, whenever $\ibox$ is a $\C$-endofunctor, i.e., is pointed (delayed) with $\point: \Id \to \ibox$ as above, $\point_M$ is a subcoalgebra of $\point_X$ for any $M \stackrel{m}{\monar} X$: 
\[
\vcenter{
\xymatrix{
M \ar@{>->} `u[r] `[rr]^-{m} [rr] \ar[dr]_-{\point_M} \ar@{.>}[r]
& \pbmod{\point}{X} M \ar[d]\ar@{>->}[r]^-{\pbmod{\point}{X} m}_<<{\!\!\!\!\!\!\!\!\!\pb} & X \ar[d]^{\point_X}\\ 
& \ibox M \ar@{>->}[r]^{\ibox m    } & \ibox\, X
}}
\]
meaning that $m \leq \pbmod{\point}{X} m$ in the natural partial order of subobjects of $X$, i.e., the ``local'' translation of $\C$ is universally valid in any $\PropS{X}$. Note that the apart from pointedness itself, we only used here the defining property of pullbacks. 
Nevertheless, transferring inhabitation laws for a pointed endofunctor $\ibox$ to inequalities holding in any $\PropS{X}$ for the corresponding modality $\pbmod{\point}{X} $ is actually a non-trivial challenge and does not always work. 

In fact, let us consider the remaining part of $\systres$, namely well-foundedness. For the poset of subobjects, this law is naturally expressed as the L\"ob rule: $\pbmod{\point}{X} m \leq m$ implies $\top \leq m$. This condition amounts precisely to the claim that $X \stackrel{\point_X}{\to} \ibox\, X$ is a \emph{well-founded coalgebra} \cite{AdamekMMS12:fossacs}.  And, as it turns out, even in natural examples of guarded fixpoint categories (\systres-categories discussed in \S\ \ref{sec:proptypes} above) the induced $\pbmod{\point}{X}$ modality may fail to be L\"ob in this sense!

\begin{example} \label{ex:cms}
Consider the guarded dagger category $(\lna{CMS},\ibox)$ of complete 1-bounded metric spaces and non-expansive mappings, where $\ibox$ is contracting all distances by a fixed factor $r$. Then $\point_X$ is just identity on points; similarly, for every morphism $m$ and every $x$, we have that $\ibox\, m(x) = m(x)$. Fix any $M \stackrel{m}{\monar} X$. The pullback $\pbmod{\point}{X} M$ is obtained set-wise, i.e., as the subset 
\[
\{(x,n) \mid x \in X, n \in M,  \point_X(x) = \ibox \, m(n)  \} 
\]
which in turn simplifies to $\{(m(n),m(n)) \mid n \in M\}$. In order words, this is just the subset of the diagonal corresponding to the image of $M$ in $X$. Crucially, as the product distance is given by supremum, it does not become contracted and the resulting space is isomorphic to (the $m$-image of) $M$. But there are non-trivial subspaces in this category, thus well-foundedness (the L\"ob rule) fails. 
\end{example}

Nevertheless, there is an important subclass of guarded dagger categories for which the induced modality on subobjects is strong L\"ob in this sense.  Recall that a cartesian closed category $\ctE$ is a(n \emph{elementary}) \emph{topos} if it comes equipped with a subobject classifier  $(\Omega, \trmo \stackrel{\topT}{\to} \Omega)$, i.e., for any monic $Y \stackrel{f}{\monar} X$ there exists exactly one mapping $X \stackrel{\chrm{f}}{\to} \Omega$ s.t. 
we have a pullback diagram:
\[
\vcenter{
\xymatrix{
Y \ar[d]_{\finim_Y}\ar@{>->}[r]^{f}_<<{\pb} & X \ar[d]^{\chrm{f}}\\ 
\trmo \ar[r]^{\topT} & \Omega
}}
\]

\noindent  
This definition already implies that the category has coproducts 
 \cite[\S\ 4.3]{Goldblatt06:topoi}. 
  Toposes are one of most important and well-studied categorical setups \cite{LambekS86:ihocl,MacLaneM92:sigl,johnstone2002sketches}. In our case, they provide precisely the missing structure to disable the pathological Example \ref{ex:cms}. 

\begin{theorem} \label{th:topolob}
Let $(\ctC,\ibox)$ be a unique guarded dagger category s.t.
\begin{itemize}
\item $\ctC$ is a topos and
\item $\ibox$ preserves  finite limits \textup(in particular monos and pullbacks\textup).
\end{itemize}
In such a situation, $X \stackrel{\point_X}{\to} \ibox X$ is a well-founded coalgebra for any $X$.  
 \end{theorem}

This is Corollary 6.10 in Birkedal et al. \cite{BirkedalMSS12:lmcs}. But since we put toposes into the picture, we have to note that in such a setting, there is yet another light in which we might see the validity of the strong L\"ob axiom or failure thereof.

\subsection{$\systres$ Modality as a Logical Connective: Scattered Toposes} \label{sec:propomega}

There is yet another logical perspective to toposes. It can be called  \emph{propositions as truth values}, and the resulting logic is higher-order. Namely, the subobject classifier is naturally seen as the object of truth values. The definition of resulting internal logic can be found, e.g., in Lambek and Scott \cite[Def. 3.5]{LambekS86:ihocl}. 

Variables and constants of a given type $A$ are seen as global elements $\trmo \to A$. Thus, elements of $\ctC[\trmo, \Omega]$ can be identified with logical constants, those of $\ctC[X,\Omega]$ ---with predicates over $X$ (i.e., formulas with a single free variable from $X$) and $n$-ary propositional connectives are elements of $\ctC[\Omega^n,\Omega]$. 
Define $\botT \deq \chrm{\finim_0}$. This yields a logical constant. Similarly, we can define logical connectives as follows: 
\begin{center}
$\negT \deq \chrm{\botT}$, \quad $\andT \deq \chrm{\prdar{\topT,\topT}}$ \quad and $\orT \deq \chrm{\cprar{\prdar{\topT_\Omega,\id_\Omega},\prdar{\id_\Omega,\topT_\Omega}}}$. 
\end{center}
For any $X\in\ctE$, $\topT_X : X \to \Omega$ stands for $\topT \cmp \finim_X$ and $\eqcT_X$ stands for  $\chrm{\prdar{\id_X,\id_X}}$. The latter yields the definition of \emph{internal equality for generalized elements of type $X$} as $\sigma \eqT \tau \deq \eqcT_X \cmp \prdar{\sigma,\tau}$. That is, for any $A \stackrel{\sigma}{\to} X$ and $B \stackrel{\tau}{\to} X$, we have $A \times B \stackrel{\sigma \eqT  \tau}{\longrightarrow} \Omega$. For $\Omega$, we can define not only $\eqcT_\Omega$, but also $\leqcT_\Omega$ as the equalizer of 
$\vcenter{
\xymatrix{
\Omega \times \Omega \ar@<1ex>[r]^-{\andT}\ar@<-1ex>[r]_-{\prdf} & \Omega
}}$. 
Implication, the only remaining intuitionistic connective, can be now defined as $\impT \deq \chrm{\leqcT_\Omega}$. 

 Here is where the higher-order logic aspect comes into play. For any type $A,$ (corresponding to an object on the categorical side) and any predicate $\phi : A \to \Omega$ over $A$, we can define $\forall x : A. \phi: \trmo \to \Omega$ as $(\lambda x : A.\phi)  \eqT (\lambda x: A. \topT_A)$. But $A$ itself can be taken to be $\Omega$: it is possible to quantify over propositions. 

 And regarding subobject and predicate transformers of \rfse{proppred}, note that the correspondence between a subobject $Y \stackrel{m}{\monar} X$ and its characteristic morphism $X \stackrel{\chrm{m}}{\longrightarrow} \Omega$ is 1-1. Starting from the opposite side: given any $X \stackrel{q}{\longrightarrow} \Omega$, one can pull it back along $\trmo \stackrel{\topT}{\longrightarrow} \Omega$ to obtain a subobject of $X$ (to show that it indeed yields a mono, use the defining properties of pullbacks and $\trmo$).

 The above pullback argument shows that any $\Omega \to \Omega$ induces uniformly on each $X$ what we called a predicate transformer or a modality on subobjects, i.e., a (set-theoretic) map from $\PropS{X}$ to $\PropS{X}$. But how and when is the opposite transformation possible? In other words, when given a generic predicate transformer, i.e., a ``dependent function'' which assigns to every $X$ a (set-theoretic) endofunction on $\PropS{X}$, can we tell whether it is induced by an element of $\ctC[\Omega,\Omega]$?  Given the topos identification of $\PropS{X}$ with $\ctC[X,\Omega]$, this is in fact a direct application of the Yoneda Lemma: such ``generic predicate transformers'' are exactly natural transformations from $\ctC[\cdot,\Omega]$ to $\ctC[\cdot,\Omega]$. And the fact that $ \pbmod{\point}{X}$ is natural in $X$ for any  (finite) limit-preserving guarded fixpoint operator $\ibox$ on a topos  is established by Proposition 6.5 and Section 6.2 of  Birkedal et al. \cite{BirkedalMSS12:lmcs}. Furthermore, their Theorem 6.8  uses the above Theorem \ref{th:topolob} (Corollary 6.10 in \emph{op.cit}) to show that in such a situation the resulting $\interm: \Omega \to \Omega$ is a $\systres$-modality in the internal logic, i.e., one satisfying 
\begin{equation} \label{eq:lobsca}
\forall \phi : \Omega. (\interm\phi \to \phi) \to \phi.
\end{equation}

Interestingly, toposes with an internal modality $\interm: \Omega \to \Omega$ satisfying \refeq{eq:lobsca} have been considered by Esakia et al. \cite{EsakiaJP00:apal} as \textit{scattered toposes}. See an overview of the second author \cite{Litak14:trends} for a more detailed discussion.

\newcommand{\etalchar}[1]{$^{#1}$}

\appendix

\section{\iglu{\C}\ in Nakano's Papers} \label{nakorig}

\newcommand{\nak}{\mathbb{L}^\to}
\newcommand{\cnak}{\mathcal{C}\mathbb{L}^\to}
\newcommand{\nakfp}{\mathbb{L}^\to_{\mathsf{fp}}}
\newcommand{\isom}{\simeq}
\newcommand{\subt}{\curlyeqprec}

Nakano's original work  \cite{Nakano00:lics,Nakano01:tacs} is of some interest also today.
 First, an important r\^ole is played therein by the existence and uniqueness (up to a suitable notion of equivalence) 
of fixpoints of  \emph{type expressions}---i.e., logical  
formulas 
 via the Curry-Howard correspondence---where the type variable is guarded 
by the modality. 
Compared to other references presented in Section \ref{sec:nakano}, Nakano's pioneering papers are also among the few ones explicitly addressing the connection with earlier modal results. For this reason, we present them in some detail in this appendix. Some caveats are necessary:
\begin{itemize}
\item The long list of references mentioned in \rfsc{tfsl} is not discussed by Nakano \cite{Nakano00:lics}, who 
says instead ``Similar results concerning the existence of fixed points
of proper type expressions could historically go back to the fixed point theorem of the logic of 
provability\dots \emph{The difference is that our logic is
intuitionistic}'' and focuses on references assuming classical propositional basis  \cite{Boolos93:lop,japa:logi98}.  On the other hand, as we will see, the connection with fixpoint theorems for L\"ob-like logics is not entirely straightforward.
\item Nakano extends the scope of his fixpoint operator (written somewhat misleadingly as $\mu$, which we correct here to our $\fip$) 
slightly beyond modalised formulas, to include as well some formulas equivalent to $\top$.
\item Furthermore, Nakano considers more fine-grained equivalence $\isom$ (respectively quasi-order $\subt$) on formulas than the one induced by  $\vdash \phi \leftrightarrow \psi$ (respectively $\vdash \phi \to \psi$).
\end{itemize}
Let us discuss these issues in more depth. 

Without the fixpoint operator, Nakano's syntax of type expressions is the implicational fragment of $\lang(\opr)$, i.e., $\nak(\opr)$ defined by $\phi$ in the following grammar

\begin{itemize}
\item
$\phi ::= \pi \mid \top \mid  (\phi \to \phi) \mid \opr\phi$,
\item $\pi ::= p_0 \mid p_1 \mid \dots$.
\end{itemize}


\noindent
A variant of Nakano's fixpoint extension\footnote{\label{ft:variables}Staying entirely faithful to Nakano's setup would require using explicit bound fixpoint variables, something we are deliberately shying away from in this paper.} can be defined very similarly to that of our \S\ \ref{fpsl} above, with a minor twist. Namely, $\nakfp(\opr)$ is defined by $\phi$ in the following grammar:

\begin{itemize}
\item
$\phi ::=  \pi \mid \top \mid \fip\chi \mid (\phi\to\phi) \mid \opr\phi$,
\item
$\pi ::= p_0 \mid p_1\mid p_2 \dots$, 
\item
$\chi ::= \pi \mid \top  \mid \fip\chi  \mid  (\chi \to \chi) \mid (\psi \to \tau) \mid \opr \psi$, 
\item
$\psi ::=  \ast \mid \pi \mid \top \mid \fip\chi \mid (\psi\to\psi) \mid \opr\psi $,
\item $\tau ::= \fip \tau \mid \opr \tau \mid \opr\tau'$,
\item $\tau' ::= \ast \mid \fip \tau' \mid \opr \tau'$.
\end{itemize}
Just like in \S\ \ref{fpsl}, we note that $\ast$ is supposed to be bound by the lowest $\fip{}$ that is above it in the parse tree. 
Other relevant remarks from \S\ \ref{fpsl} apply here as well.  The (small) novelty here consists in the clause $\psi \to \tau$ in $\chi$; the class of expressions defined by $\tau$ are called \emph{$\top$-variants} by Nakano \cite{Nakano00:lics}.\footnote{In the subsequent reference 
\cite{Nakano01:tacs}, a distinction is made between  \emph{$F\!\hyph\!\top$-variants} and \emph{$S\!\hyph\!\top$-variants}; 
the notion introduced above corresponds to the former one. Motivation comes from the intended semantics of $\lambda$-expressions.} 
He proceeds then to define an equivalence relation $\isom$ on $\nakfp(\opr)$ as the smallest one which is:
\begin{itemize}
\item congruent wrt $\opr$ and $\to$,
\item satisfying $(\phi \to \top) \isom (\phi' \to \top)$,
\item and congruent wrt (un)folding laws for the unique fixpoint operator $\fip$, i.e., 
\begin{itemize}
\item $\fip\chi  \isom \chi\fip\chi$ and 
\item $\phi \isom \chi\phi$ implies $\phi \isom \fip\chi$.
\end{itemize}
\end{itemize}
Nakano observes that terms $\isom$-equivalent to $\top$ are precisely $\top$-variants, i.e.,
 those given by $\tau$ in the above grammar. Furthermore, while $\isom$ is more fine-grained that provable equivalence 
 (as we are going to see below), this is the level on which Nakano incorporates a very restricted version of the syntactic fixpoint construction, 
 basically reflecting the fact that $\top$-variants are equivalent to $\top$. Namely, the set of \emph{canonical} formulas 
 (type expressions) $\cnak(\opr)$ is defined by $\alpha$ in the following grammar:
\begin{itemize}
\item $\alpha ::= \opr\alpha \mid \beta$,
\item $\beta ::= \pi \mid \top \mid  (\phi \to \phi)$,
\end{itemize}
and $\phi$ is as given above in the definition of $\nakfp(\opr)$. In particular, in formulas from $\cnak(\opr)$, the operator $\fip$ can only occur within the scope of $\to$. Nakano shows that every formula in $\nak(\opr)$ is $\isom$-equivalent to one in $\cnak(\opr)$ via a translation inductively replacing $\fip\Box^n\ast$ and $\Box^n\top$ with $\top$.


Nakano's papers define several variants of the subtyping relation\footnote{We change Nakano's notation here to avoid clash with the notation we use in presenting Kripke semantics.} $\subt$. Their common core is  the smallest transitive relation on $\nakfp(\opr)$ containing $\isom$ and furthermore:
\begin{itemize}
\item antimonotone in the antecedent and monotone in the consequent of $\to$,
\item congruent\footnote{As mentioned in Footnote \ref{ft:variables}, our setup simplifies that of Nakano inasmuch as fixpoint (type) variables are concerned. This matters when defining the interaction of $\subt$ with $\fip$ in a (sub)typing context: type variables bound by fixpoint operators might be related by $\subt$ and Nakano has an explicit rule dealing with this.} wrt $\opr$ and $\fip$,
\item with $\top$ as the greatest element, i.e., $\phi \subt \top$,
\item satisfying $\phi \subt \opr\phi$ and $\phi \to \psi \subt \opr\phi \to \opr\psi$.
\end{itemize}
 
 These are the minimal subtyping rules. Furthermore,
 \begin{description}
 \item[Nak0] one possible extension
 \cite{Nakano00:lics} involves $\opr\phi \to \opr\psi \subt \Box(\phi \to \psi)$;
 \item[Nak1] another one \cite{Nakano01:tacs} involves $\Box(\phi \to \psi) \subt \opr\phi \to \opr\psi$.
 \end{description}
But $\subt$ is still more fine-grained than the relation of provable implication. Namely, relatively to a given subtyping relation $\subt$, Nakano defines an extension of (sub)typed lambda calculus that apart from standard STLC rules and the subtyping rule includes specifically modal rules
\begin{itemize}
\item From $\varGamma \vdash M: \phi$ infer $\opr\varGamma \vdash M: \opr\phi$ and 
\item From $\varGamma \vdash M: \Box^n(\phi \to \psi)$ and $\varGamma \vdash N: \Box^n\phi$ infer $\varGamma \vdash MN: \Box^n\psi$.
\end{itemize}

Clearly,  Nakano's logic of \emph{type inhabitation} is stronger than the logic of subtyping: $\phi \subt \psi$ implies that $\phi \to \psi$ is inhabited, but not necessarily the other way around. In particular, Nakano's typing rules entail that the logic of inhabitation is always normal, whether or not normality is built into the subtyping relation (i.e., even in the absence of Nak1 above). Of course, the difference between the logic of inhabitation and the logic of subtyping is visible even in the absence of modality: well-known references \cite{BarendregtEA83:jsl,DezaniEA02:ndjfl,DezaniEA03:entcs} connect subtyping with weak relevance logics.

In the absence of fixpoint operators, Nakano's logic of type inhabitation is simply $\iglu{\C}$ restricted to $\nak(\opr)$ \cite{Nakano01:tacs} or an extension thereof if Nak0 above is incorporated in the subtyping discipline \cite{Nakano00:lics}. However, Nakano is interested in semantics of his $\lambda$-calculus and equality of types (corresponding to $\isom$ above) does not coincide with provable equivalence. The latter  means that one type is inhabited iff the other is, which is a fairly coarse relation. But the syntactic fixpoint construction for L\"ob-like logics works modulo provable equivalence.  
 Instead, Nakano essentially relies on the  existence of  Banach-style fixpoints in ultrametric spaces, and numerous later references proceed in an analogous way.

\end{document}